\documentclass[11pt, a4paper]{amsart}
\usepackage{amscd, amssymb, pb-diagram,tikz}
\usepackage{amsthm}
\usepackage[shortlabels]{enumitem}
\usepackage[citecolor = blue, colorlinks = true, linkcolor = blue, linktoc = all, pdffitwindow = false, urlcolor = cyan,bookmarks=false,bookmarksdepth={1}, bookmarksopenlevel=section]{hyperref}
\def\End{\operatorname{End}}

\def\id{\operatorname{id}}

\def\max{\operatorname{max}}

\def\id{\operatorname{id}}

\def\clsp{\overline{\operatorname{span}}}

\def\C{\mathbb{C}}
\def\F{\mathbb{F}}

\def\N{\mathbb{N}}
\def\Z{\mathbb{Z}}

\def\P{\mathbb{P}}

\def\Q{\mathbb{Q}}

\def\TT{\mathcal{T}}

\def\OO{\mathcal{O}}
\def\KK{\mathcal{K}}

\def\DD{\mathcal{D}}

\def\CC{\mathcal{C}}

\def\SS{\mathcal{S}}

\def\JJ{\mathcal{J}}
\def\QQ{\mathcal{Q}}

\def\gxp{G\rtimes_\theta P}

\def\nxnx{\mathbb{N}\rtimes\mathbb{N}^\times}

\newcommand{\Chi}{\raisebox{2pt}{\ensuremath{\chi}}}

\newtheorem{thm}{Theorem}[section]
\newtheorem{cor}[thm]{Corollary}

\newtheorem{lemma}[thm]{Lemma}

\newtheorem{prop}[thm]{Proposition}

\theoremstyle{definition}
\newtheorem{definition}[thm]{Definition}

\newtheorem{notation}[thm]{Notation}

\theoremstyle{remark}
\newtheorem{remark}[thm]{Remark}

\newtheorem{example}[thm]{Example}

\numberwithin{equation}{section}

\begin{document}

\title[On $C^*$-algebras associated to right LCM semigroups]{On $C^*$-algebras associated to right LCM semigroups}

\author{Nathan Brownlowe}
\address{School of Mathematics and Applied Statistics \\ University of Wollongong \\ Australia}
\email{nathanb@uow.edu.au}

\author{Nadia S.~Larsen}
\address{Department of Mathematics \\ University of Oslo \\P.O. Box 1053, Blindern\\0316 Oslo\\ Norway}
\email{nadiasl@math.uio.no}

\author{Nicolai Stammeier}
\address{Mathematisches Institut, Westf\"{a}lischen Wilhelms-Universit\"{a}t M\"{u}nster \\ Germany}
\email{n.stammeier@wwu.de}

\thanks{Part of this research was carried out while all three authors participated in the workshop "Operator algebras and dynamical systems from number theory" in November 2013 at the Banff International Research Station, Canada. We thank BIRS for hospitality and excellent working environment. The third author was supported by DFG through SFB $878$ and by ERC through AdG $267079$.}

\begin{abstract} We initiate the study of the internal structure of $C^*$-algebras associated to a left cancellative semigroup in which any two principal right ideals are either disjoint or intersect in another principal right ideal; these are variously called right LCM semigroups or semigroups that satisfy Clifford's condition. Our main findings are results about uniqueness of the full semigroup $C^*$-algebra. We build our analysis upon a rich interaction between the group of units of the semigroup and the family of constructible right ideals. As an application we identify algebraic conditions on $S$ under which $C^*(S)$ is purely infinite and simple.
\end{abstract}

\date{22 June 2014. Revised on 24 November 2014.}
\maketitle

\section{Introduction}\label{sec: intro}
\noindent In recent years, $C^*$-algebras associated to semigroups have received much
attention due to the
range of new examples and interesting applications that they encompass. One
such application is to the connections between operator algebras and number
theory, which have grown deeper since Cuntz's work in \cite{Cun1} on the
$C^*$-algebra $\QQ_\N$ associated to the affine semigroup over the natural
numbers $\nxnx$. Laca and Raeburn \cite{LacR2} continued the
analysis of $C^*$-algebras associated to $\nxnx$ by examining the Toeplitz
algebra $\TT(\nxnx)$, including an analysis of its KMS structure. Cuntz,
Deninger and Laca \cite{CDL} have since examined the KMS structure of
Toeplitz-type
$C^*$-algebras associated to $ax+b$-semigroups $R\rtimes R^\times$ of rings of
integers $R$ in number fields.

Li has recently defined $C^*$-algebras associated to left cancellative
semigroups $S$ with identity, and initiated a study of when certain naturally
arising $*$-homomorphisms are injective \cite{Li1, Li2}. The reduced
$C^*$-algebra $C_r^*(S)$ associated to $S$ is defined by means of the left regular
representation of $S$ on the Hilbert space $\ell^2(S)$. The full $C^*$-algebra
$C^*(S)$ is defined to be the universal $C^*$-algebra generated by isometries
and projections, subject to certain relations which are imposed by the regular
representation. For certain classes of semigroups, the
canonical isomorphism between the full and reduced semigroup $C^*$-algebras
was established in \cite{Li1, Li2, No0}.

In \cite{BRRW}, the authors studied the full semigroup $C^*$-algebra arising from an algebraic construction
called a Zappa-S{\'z}ep product of semigroups. The resulting semigroups display ordering features similar to
the quasi-lattice ordered semigroups introduced by Nica \cite{Ni}, but by contrast contain a non-trivial group of
units. These semigroups were called right LCM (for least common multiples) in \cite{BRRW}, and we
shall henceforth use this terminology, but mention  that in \cite[\S 4.1]{Law2} and \cite{No0} these are known as
semigroups that satisfy Clifford's condition. The class of right LCM semigroups is pleasantly large and includes quasi-lattice
ordered semigroups, certain semidirect products of semigroups, and also semigroups that model self-similar group actions, see \cite{Law1, LW, BRRW}.

In the present work we begin a study of the internal structure of $C^*$-algebras
associated to right LCM semigroups. The main thrust of our work is that when $S$ is a right LCM semigroup one may unveil the internal structure of $C^*(S)$ and answer questions about its uniqueness by carefully analysing the relationship between the group of units
$S^*$  and the constructible right ideals of $S$.

The problem of finding good criteria for injectivity of $*$-homomorphisms on
$C^*(S)$ and in particular to decide uniqueness of such $C^*$-algebras is at the moment not settled in the generality of
left cancellative semigroups. A powerful method to prove injectivity of
$*$-representations was developed by Laca and Raeburn in \cite[Theorem 3.7]{LaRa} for
$C^*(S)$ with $(G, S)$ quasi-lattice ordered.
Their work recasted Nica's $C^*$-algebras associated to quasi-lattice ordered
groups in \cite{Ni} by viewing them as $C^*$-crossed products by semigroups of
endomorphisms. Based on this realisation, they adapted a technique introduced
by Cuntz in \cite{Cun} which involved expecting onto a diagonal subalgebra.

There are new technical obstacles to be overcome when dealing with a semigroup
$S$ that has a non-trivial group of units. In
particular, not all of Laca and Raeburn's programme can be carried through
beyond the case of quasi-lattice ordered pairs. One challenge is that the
diagonal subalgebra of $C^*(S)$, denoted $\DD$ in \cite{Li1}, may be too small
to accommodate the range of a conditional expectation from $C^*(S)$, cf. an
observation made in \cite{No0}. Furthermore, generating isometries in $C^*(S)$
that correspond to elements from the group of units $S^*$ give rise to
unitaries. These unitaries together with the generating projections from
$\DD$ yield two new subalgebras of $C^*(S)$ whose role in explaining the
structure of $C^*(S)$ is yet to be fully understood.

Our initial approach was to push to the fullest extent the Laca-Raeburn
strategy to an arbitrary right LCM semigroup $S$, with or without an identity.
It soon became evident that the presence of non-trivial units in $S^*$ makes it
unlikely that \cite[Theorem 3.7]{LaRa} will extend in the greatest generality
to right LCM semigroups. However, by carefully analysing the action of the group of units $S^*$ on the
constructible right ideals $\JJ(S)$ of $S$ we are able to identify conditions on $S$ which ensure that
injectivity of $*$-homomorphisms on $C^*(S)$ can be characterised on $\DD$. This approach has lead us to
find conditions on a right LCM
semigroup $S$ which ensure that $C^*(S)$ is purely infinite and simple. The examples we
have of such semigroups belong to a class of semidirect products $\gxp$ of a
group $G$ by an injective endomorphic action $\theta$ of a semigroup $P$.
$C^*$-algebras associated to such semidirect products where $P=\mathbb{N}$
were studied by Cuntz and Vershik in \cite{CV}, and  by Vieira in
\cite{Vie}. Our $C^*(\gxp)$ may be interpreted as higher dimensional versions
of those  $C^*$-algebras. We mention that K-theory and internal structure of $C^*$-algebras
associated to $ax+b$-semigroups of certain integral domains were analysed recently by Li, see \cite{Li3}.

The organisation of the paper is as follows. In Section~\ref{sec: background}
we collect some standard results about semigroups. We also introduce our
conventions on semidirect product semigroups, and identify an abstract
characterisation of the examples of interest $\gxp$. Section~\ref{sec: right
LCM algebras} contains an introduction to right LCM semigroups, and their
associated full and reduced $C^*$-algebras. Since we do not assume that $S$
necessarily contains an identity element, we explain how the definitions of
$C_r^*(S)$ and $C^*(S)$ from \cite{Li1} can be adapted to this, slightly more
general, situation. In the same section we introduce the distinguished
subalgebras of interest, which are built out of $\DD$ and the unitaries coming
from the group of units $S^*$. We also discuss conditional expectations onto
the diagonal subalgebras of $C^*(S)$ and $C_r^*(S)$.

Our first findings about injectivity of a $*$-homomorphism on $C^*(S)$ are the
subject of Section~\ref{sec: diagonalsection}. We show in Theorem~\ref{thm:
using the diagonal} that injectivity can be phrased as a nonvanishing
condition involving projections from $\DD$, similar to \cite[Theorem 3.7]{LaRa}, when the semigroup $S$ has at most an identity element as unit,
or, in the presence of non-trivial units, satisfies a technical condition on the left action of $S^*$ on the space $\JJ(S)$. In Section~\ref{section:pisimple} we identify
a number of conditions on a right LCM semigroup $S$ which imply that $C^*(S)$
is purely infinite and simple. These conditions include a characterisation of
the left action of  $S^*$ on $\JJ(S)$
which is a refined version of effective action; we call this strongly
effective. In a short Section~\ref{section:injregular} we discuss injectivity of the
canonical surjection from $C^*(S)$ onto $C_r^*(S)$ and illustrate this with
semigroups of the form $\gxp$. Section~\ref{section:usingcore} initiates the
study of injectivity of $*$-homomorphisms on $C^*(S)$ phrased in terms of a
core subalgebra that is built from $\DD$ and the unitaries corresponding to
the group of units $S^*$ in $S$. The final section, Section~\ref{sec: examples}, is
devoted to applications. Here we discuss the validity of the properties of
right LCM semigroups introduced in sections~\ref{sec: diagonalsection} and \ref{section:pisimple}. The main class of
examples is that of semidirect products of the form $\gxp$, and via
Theorem~\ref{thm:gxp p.i. simple} we provide examples of purely infinite
simple $C^*(S)$ from this class. We also take the opportunity to examine the
Zappa-Sz\'ep product semigroups $X^*\bowtie G$ coming from self-similar
actions $(G,X)$ as considered in \cite{Law1, LW, BRRW}; in particular, we
examine some of the properties of semigroups introduced in the paper. While at this stage
we
cannot apply our $C^*$-algebraic results to this class of semigroups, we
plan to examine these problems in further work.

We thank the referee for suggesting many improvements to the presentation.

\section{Some results on semigroups}\label{sec: background}
\noindent By a semigroup $S$ we understand a non-empty set $S$ with an associative
operation. We refer to \cite{Cli-Pre} and \cite{Lal} for basic properties
of semigroups. Semigroups with an identity element for the operation are
known as monoids. Here we shall use the terminology semigroup, and specify
existence of an identity when this is the case.
All semigroups considered in this work are discrete. A semigroup $S$ is {\em
left cancellative} if $pq=pr$ implies $q=r$ for
all $p,q,r\in S$; {\em right cancellative} if $qp=rp$ implies $q=r$ for
all $p,q,r\in S$; and {\em cancellative} if it is both left and right
cancellative.

Given a semigroup $S$ with identity $1_S$, an element $x$ in $S$ is invertible
if there is $y\in S$ such that $xy=yx=1_S$. We denote by $S^*$ the group of
invertible elements of $S$ (also called the group of units of $S$). We shall
write $S^*\neq \emptyset$ in case the group of units is non-trivial (possibly
consisting only of the identity element), and we write $S^*=\emptyset$
otherwise.  If $S$ is cancellative and $x\in S^*$, then $x^{-1}$ will denote the inverse of $x$.

The  Green relations on a semigroup are well-known, see for example
\cite[Chapter 2]{Lal}. The left Green relation $\mathcal{L}$ is $a\mathcal{L}b$ if and only
if $Sa=Sb$ for $a,b\in S$. Likewise, the right Green relation $\mathcal{R}$ is  given by
$a\mathcal{R}b$ if and only if $aS=bS$ for $a, b\in S$. Suppose that $S$ is a
semigroup with $S^* \neq \emptyset$.  Since $Sx=S$ whenever $x\in S^*$, we see
that $a=xb$  for some $x\in S^*$ implies that $a\mathcal{L}b$. If $S$ is right
cancellative, the reverse implication holds and, moreover, the element $x$ in
$S^*$ is unique. Indeed, let $Sa=Sb$. Then there are $c,d\in S$ such that
$b=ca$ and $a=db$, so $b=cdb$ and $a=dca$. Thus right cancellation implies
$cd=1_S=dc$, showing that $c,d\in S^*$. If right cancellation is replaced with
left cancellation in the previous considerations, then $a\mathcal{R}b$ is the
same as $a=by$ for a unique $y\in S^*$.

If $S^*= \emptyset$, we will assume throughout this paper that $S$ has the following property: if $a,b \in S$ satisfy $aS = bS$, then $a=b$. This is what happens in the case that $S^* = \{1_S\}$.

 Given a semigroup $S$, a right ideal $R$ is a non-empty subset  of $S$ such
 that $RS\subseteq R$. The \emph{principal right ideals} of $S$ are all the
 right ideals of the form $pS:=\{ps\mid s \in S\}$ for $p\in S$. Given a principal right ideal $pS$, an element
 $r\in pS$ is called a \emph{right multiple} of $p$. The right ideal
 generated by $p\in S$ is defined as $\{p\}\cup pS$; we shall denote it
 $\langle p\rangle$.

 \begin{remark}\label{rmk:ideal-and-regular}
  If $S$ has an identity it is clear that $pS=\langle p\rangle$. For an arbitrary left cancellative semigroup $S$ and $p\in S$, a sufficient condition to have $pS=\langle p\rangle$ is that there is an idempotent $t\in S$, i.e. $t=tt$, such that $p=pt$. Note that if $p$ is a regular element of $S$, in the sense that there is $s\in S$ such that
$p=psp$, then $t=sp$ is an idempotent such that $p=pt$. Thus $p\in pS$ whenever $p$ is a regular element in a semigroup $S$.
\end{remark}

\begin{definition}\label{def:right-LCM-sgp}
A semigroup $S$ is {\em right LCM} if
it is left cancellative and every pair of elements $p$ and $q$ with a right
common multiple has a right least common multiple $r$.
\end{definition}

\noindent It is clear that a semigroup $S$ is right LCM if it is left cancellative and
for any $p, q$ in $S$, the intersection
of principal right ideals  $pS\cap qS$ is either empty or of the form $rS$ for
some $r\in S$. This property of semigroups is called {\em Clifford's condition}
in \cite[\S 4.1]{Law2} and \cite{No0}. In general, right least common multiples are not
unique: if $r$ is a right least common multiple of $p$ and $q$, then so is $rx$ for any
$x\in S^*$.

The quasi-lattice ordered groups treated in \cite{Ni} are examples of right LCM
semigroups with unique right least common multiples. We discuss
other  examples in Section~\ref{sec: examples}.
The main class of examples of semigroups that is considered in the present
work is that of semidirect product semigroups. We introduce next our
conventions for a semidirect product of semigroups.

For a semigroup $T$ we let $\End{T}$ denote the semigroup of all homomorphisms
$T\to T$. The identity endomorphism is $\id_T$. An action $P
\stackrel{\theta}{\curvearrowright} T$ of a semigroup $P$ on $T$ is a
homomorphism $\theta:P\to \End{T}$, i.e. $\theta_p\theta_q=\theta_{pq}$
for all $p,q\in P$.
If $T$ has an identity $1_T$, we shall require that
$\theta_p(1_T)=1_T$ for all $p\in P$. In case $P$ has an identity $1_P$, we shall further require that
$\theta_{1_P}$ is the identity endomorphism of $T$.

\begin{definition}\label{def:sdps} Let $T, P$ be semigroups and $P
\stackrel{\theta}{\curvearrowright} T$ an action.
The \emph{semidirect product} of $T$ by $P$ with respect to $\theta$, denoted
$T{\rtimes_\theta}P$, is the semigroup $T\times P$ with composition given by
\[
(s,p)(t,q) = (s\theta_{p}(t),pq),
\]
for $s,t\in T$ and $p, q\in P$.
\end{definition}

\noindent Examples of semidirect products are $ax+b$-semigroups, where $T$ comprises the
additive structure, and $P$ the multiplicative structure in some ring or
field. It is known that $T\rtimes_\theta P$ is right cancellative when $T$ and $P$ are both right
cancellative, and $T\rtimes_\theta P$ is left cancellative when $T$ and $P$
are both left cancellative and, in addition, $\theta$ is an action by
injective endomorphisms of $T$.

In the next result we describe $S^*$ in the case of a
semidirect product $S=\gxp$ in which $G$ is a group.

\begin{lemma}\label{lem: units in GxP}
Let $G$ be a group, $P$ a semigroup and $P
\stackrel{\theta}{\curvearrowright}G$ an action such that $\gxp$ is left
cancellative. If $P$ has an identity, then $(\gxp)^* = G\rtimes_{\theta}P^*$
holds, otherwise $\gxp$ does not have an identity.
\end{lemma}
\begin{proof}
If $P$ has an identity element $1_P$, the identity element of $\gxp$ is given
by $(1_G,1_P)$. Now let $(g,x) \in (\gxp)^*$. By definition, there is $(h,y)
\in \gxp$ such that $(g\theta_x(h),xy) = (g,x)(h,y) = (1_G,1_P)$. Thus, $x \in
P^*$. Conversely, if $x \in P^*$ and $g \in G$, the inverse of $(g,x)$ is
given by $(\theta_{x^{-1}}(g^{-1}),x^{-1})$. The second case is obvious.
\end{proof}

\begin{remark} Let $G$ be a group, $P$ a semigroup with $P^*=\{1_P\}$ and $P
\stackrel{\theta}{\curvearrowright}G$ an action such that $\gxp$ is left
cancellative. Given $(g,p)\in \gxp$, we have
$(g,p)(h,1_P)=(g\theta_p(h)g^{-1},1_P)(g,p)$ for any $h\in G$. By
Lemma~\ref{lem: units in GxP}, $a(\gxp)^*\subset (\gxp)^*a$ for any $a$ in
$\gxp$. This observation motivates the next considerations.
\end{remark}

\noindent In \cite[\S 10.3]{Cli-Pre}, a subset $H$ of a semigroup $S$ is called centric
if $aH=Ha$ for
every $a\in S$. For a semigroup $S$ with $S^*\neq\emptyset$, we
shall consider two one-sided versions of this condition.

\begin{definition} Given a semigroup $S$ with $S^*\neq\emptyset$,
let (C1) and (C2) be the conditions:
\begin{enumerate}
\item[(C1)]\label{def: centric} $aS^*\subseteq S^*a$ for all $a\in S$.
\item[(C2)]\label{def: centric-other} $S^*a\subseteq aS^*$ for all $a\in S$.
\end{enumerate}
\end{definition}

\begin{prop}\label{prop:congruence} Let $S$ be a semigroup with $S^* \neq
\emptyset$. Consider the equivalence relation on $S$ given as follows:  for
$a,b\in S$,
\[ a\sim b \text{ if }a=xb \text{ for some }x\in S^*.\]
If $S$ satisfies (C1), then $\sim$ is a congruence on $S$. Consequently, if
$\mathcal{S}:=S/_\sim$ denotes the collection of equivalence classes
$[a]:=\{b\in S\mid b\sim a\}$, then $\mathcal{S}$ is a semigroup with identity
$[1_S]$. Moreover, $\mathcal{S}^*=\{[1_S]\}$.
\end{prop}
\begin{proof}[Proof of Proposition~\ref{prop:congruence}.] It is routine to
check that $\sim$ is an equivalence relation. To show that it is a congruence on
$S$, we must show that whenever $a\sim b$ then $cad\sim cbd$ for all $c,d$ in
$S$. Let $x$ in $S^*$ such that $a=xb$.  By (C1), there is $x'\in S^*$ such
that $cx=x'c$. Then $cad=cxbd=x'cbd$, giving the claim. Thus $[a_1]\cdot
[a_2]:=[a_1a_2]$ for $a_1,a_2 \in S$ is a well-defined operation  which turns
$\mathcal{S}$ into a semigroup with identity $[1_S]$.

Suppose that $[a][b]=[1_S]=[b][a]$ for $a,b\in S$. Then $ab=x$ and $ba=y$ for
$x,y\in S^*$, which shows that
$bx^{-1}=y^{-1}b$ is an inverse for $a$. Similarly, $b\in S^*$, and thus
$[a]=[b]=[1_S]$.
\end{proof}

\begin{remark}
The relation $\sim$ from Proposition~\ref{prop:congruence} is closely related
to the left Green relation:  since $Sx=S$ whenever $x\in S^*$, we see that
$a\sim b$ implies $a\mathcal{L}b$. If $S$ is right cancellative, then also
$a\mathcal{L}b$ implies $a\sim b$.
\end{remark}

\noindent Our interest is in semigroups $S$ that are left cancellative and often
cancellative. So we would like to know when the semigroup $\mathcal{S}$ from
Proposition~\ref{prop:congruence} inherits these properties. One sufficient
condition for left cancellation to pass from $S$ to $\mathcal{S}$ is spelled
out in the next lemma, whose immediate proof we omit.

\begin{lemma}\label{lemma:calS-cancellative} Let $S$ be a semigroup with $S^*
\neq \emptyset$ and satisfying (C1). If $S$ is right cancellative then
$\mathcal{S}$ is right cancellative. Further, $\mathcal{S}$ is left
cancellative if $S$ is left cancellative and has the following property:
\begin{equation*}\label{eq:condC3}
 ab=xac \text{ for }a,b,c\in S, x\in S^* \Longrightarrow \exists
y\in S^* \text{ with }xa=ay.
\end{equation*}
\end{lemma}

\begin{prop}\label{prop:calS-for-GP} Let $P$ be a semigroup with $P^* \neq
\emptyset$, $G$ a group and $P \stackrel{\theta}{\curvearrowright} G$ an
action by injective group endomorphisms of $G$. Denote $S=\gxp$ the resulting semidirect product.

\textnormal{(a)} If $P$ satisfies (C1), then so does $S$.

\textnormal{(b)} If $P$ is right cancellative and satisfies (C1), then
$\mathcal{S}$ is right cancellative.

\textnormal{(c)} If $P$ is left cancellative and $P^*$ is centric, then $\mathcal{S}$
is left cancellative.
\end{prop}
\begin{proof} For (a), let $(g,p)\in S$ and $(g', x)\in S^*=G\rtimes P^*$,
according to Lemma~\ref{lem: units in GxP}. Choose by (C1) an element $y\in
P^*$ such that $px=yp$. It follows that
$(g, p)(g', x)=(g'', y)(g,p)$ for $g''=g\theta_p(g')\theta_y(g^{-1})$. For
assertion (b), note that $S$ has (C1) by (a) and is right cancellative, so the claim  follows by
Lemma~\ref{lemma:calS-cancellative}.

To prove (c), first note that $\mathcal{S}$ is well-defined since $S$ has
(C1).  Suppose we have elements $(g,p)$, $(h, q)$, $(k, r)$ in $S$ and
$(g_0,p_0)\in S^*$ such that $(g, p)(h,q)=(g_0, p_0)(g,p)(k, r)$. Therefore
$(g\theta_p(h), pq)=(g_0\theta_{p_0}(g\theta_p(k)),p_0pr)$. Since  $P^*$ is
centric,  there is a unique $p_1 \in P^*$ such that
$p_0p=pp_1$. Choosing $g_1=h\theta_{p_1}(k^{-1})$ in $G$ we have $(g_0,
p_0)(g,p)=(g,p)(g_1, p_1)$. Hence Lemma~\ref{lemma:calS-cancellative} applies
and shows that $\mathcal{S}$ is left cancellative.
\end{proof}

\noindent The next result shows that cancellative semigroups which are semidirect
products
of the form $\gxp$, with $P^*=\{1_P\}$, can be characterised abstractly as
cancellative semigroups
$S$ that satisfy (C1) and for which the quotient map of $S$ onto $\mathcal{S}$
admits a homomorphism lift.

\begin{prop}\label{prop: correspondence between S and GxP}
There is a bijective correspondence between the class of
cancellative semigroups $S$ with identity $1_S$ satisfying (C1) and such that the quotient map
from $S$ onto $\mathcal{S}$ admits a transversal homomorphism which embeds $\SS$ into $S$,
and the class of semidirect product semigroups $G\rtimes_\theta P$ arising
from a cancellative semigroup $P$ with $P^*=\{1_P\}$, which
acts by
injective endomorphisms of a group $G$.
\end{prop}
\begin{proof}
Suppose $S$ is cancellative with $1_S$, satisfies (C1), and is such that there is an embedding of
$\SS$ as a subsemigroup of $S$ which is a right inverse for the quotient map $S\to \mathcal{S}$. For ease of notation, we identify
$\SS\subseteq S$. Then
for each $p\in\SS$ we have a map $\theta_p:S^*\to S^*$, where
$\theta_p(x)$ is the unique element of $S^*$ satisfying $px=\theta_p(x)p$.
Note that such an element exists because of (C1), and is unique because $S$
is right cancellative. We claim that $\theta:p\mapsto \theta_p$ is an action
of $\SS$ by injective endomorphisms of $S^*$. For each $p\in \SS$ and $x,y\in
S^*$ we have $\theta_p(xy)p=pxy=\theta_p(x)py=\theta_p(x)\theta_p(y)p$, which
by right cancellation means $\theta_p(xy)=\theta_p(x)\theta_p(y)$. Since we
obviously have $\theta_p(1_S)=1_S$, each $\theta_p$ is an endomorphism of
$S^*$. For each $p,q\in\SS$ and $x\in S^*$ we have
$\theta_{pq}(x)pq=pqx=p\theta_q(x)q=\theta_p(\theta_q(x))pq$, which by right
cancellation means $\theta_{pq}(x)=\theta_p(\theta_q(x))$, and so $\theta$
is an action. Hence we can form the semidirect product $S^*\rtimes_\theta\SS$.
We have each $\theta_p$ injective because
$\theta_p(x)=\theta_p(y)$ implies $px=\theta_p(x)p=\theta_p(y)p=py$, resulting
in $x=y$.

The map $\phi:S^*\rtimes_\theta\SS\to S$ given by $\phi((x,p))=xp$ is a
homomorphism because
\[\phi((x,p))\phi((y,q))=xpyq=x\theta_p(y)pq=\phi((x,p)(y,q)).\]
For each $r\in
S$ we choose $p\in\SS$ the representative of $r$ in $\SS$. Then $r=xp$ for
some $x\in S^*$, which means $r=\phi((x,p))$, and hence $\phi$ is surjective.
For injectivity note that $\phi((x,p))=\phi((y,q))$ means $p$ and $q$ differ
by a unit. Hence as elements of $\SS$ they must be equal. Then right
cancellation gives $x=y$. So $\phi:S^*\rtimes_\theta\SS\to S$ is an
isomorphism. Moreover, $\SS$ is cancellative because $S$ is cancellative, and
we have $\SS^*=\{1_S\}$ because $[x]=[1_S]$ for all $x\in S^*$. Since
$\SS^*=\{1_S\}$, $\SS$ trivially satisfies (C1).

Now suppose that $P$ is cancellative with $P^*=\{1_P\}$, and acts by
injective endomorphisms on a group $G$. Then we know from the discussion on
semidirect products prior to Lemma~\ref{lem: units in GxP} that
$G\rtimes_\theta P$ is
cancellative. We also know from Proposition~\ref{prop:calS-for-GP} that
$G\rtimes_\theta P$ satisfies (C1). Denote by $\SS_{G,P}$ the semigroup
obtained by applying Proposition~\ref{prop:congruence} to $G\rtimes_\theta P$,
and consider the map $\pi:\SS_{G,P}\to G\rtimes_\theta P$ given by
$\pi([(g,p)])=(1_G,p)$. Since $(G\rtimes_\theta P)^*=G\times\{1_P\}$, the equality
$[(g,p)]=[(h,q)]$ implies $p=q$, which means $\pi([(g,p)])=\pi([(h,q)])$.
So $\pi$ is well defined. We have
\[\pi([(g,p)][(h,q)])=\pi([(g\theta_p(h),pq)])=pq=\pi([(g,p)])\pi([(h,q)])\]
for
each $[(g,p)],[(h,q)]\in\SS_{G,P}$, and so $\pi$ is a homomorphism. Moreover,
$\pi$ is obviously unital. Finally, for each $[(g,p)],[(h,q)]\in\SS_{G,P}$ we
have
\[\pi([(g,p)])=\pi([(h,q)])\Longrightarrow p=q,\]
so $(g,p)=(gh^{-1},1_P)(h,q)$, resulting in $[(g,p)]=[(h,q)]$.
Thus $\pi$ is injective, and hence a semigroup embedding in $G\rtimes_\theta
P$.
\end{proof}

\section{Right LCM semigroup C*-algebras}\label{sec: right LCM algebras}
\subsection{Semigroup C*-algebras}\label{subsec: semigroup algebras}~\\
In \cite{Li1}, Li constructed the reduced and the full $C^*$-algebras $C_r^*(S)$ and $C^*(S)$ associated to a left cancellative semigroup $S$ with identity. In this work we shall allow semigroups that do not necessarily have an identity, so we start by investigating to what extent the construction of $C_r^*(S)$ and $C^*(S)$ from \cite{Li1} still makes sense.

Let $S$ be a left cancellative semigroup, and let $\{\varepsilon_t\}_{t\in S}$ denote the canonical orthonormal basis of $\ell^2(S)$ such that $(\varepsilon_s| \varepsilon_t)=\delta_{s,t}$ for $s, t\in S$.  For each $p\in S$ let $V_p$ be the operator in  $\mathcal{L}(\ell^2(S))$ given by $V_p\varepsilon_t=\varepsilon_{pt}$ for all $t\in S$. We have $V_p^*V_p=I$ in $\mathcal{L}(\ell^2(S))$, so that $V_p$ is an isometry for every $p\in S$. We define the \emph{reduced} $C^*$-algebra $C_r^*(S)$ to be the unital $C^*$-subalgebra of $\mathcal{L}(\ell^2(S))$ generated by $V_p$ for all $p\in S$.

Given $p\in S$, clearly $V_pV_p^*\varepsilon_s=0$ when $s\notin pS$.  Left cancellation implies that $V_pV_p^*\varepsilon_s=\varepsilon_s$ when $s\in pS$. Thus the range projection $V_pV_p^*$ of $V_p$ is the orthogonal projection onto the subspace $l^2(pS)$ corresponding to the principal right ideal $pS$. We shall denote this projection by $E_{pS}$. With reference to Remark~\ref{rmk:ideal-and-regular}, note that $p$ need not belong to $pS$. However, $p$ is contained in $pS$ if $S$ has an identity or if $p$ is a regular element of $S$. We summarise some properties of the elements $V_p$ and $E_{pS}$ in the next lemma, whose proof we omit.

\begin{lemma}\label{lem:def-V-E-reduced-alg} Let $S$ be a left cancellative semigroup that does not
necessarily have an identity. Then for each $p$ in $S$, the range projection
of $V_p$ is equal to the orthogonal projection $E_{pS}$ onto the subspace
$l^2(pS)$. Further, the isometries $V_p$ and the projections $E_{pS}$ satisfy
the relations:
\begin{enumerate}
\item $V_pV_q=V_{pq}$;
\item $V_p E_{qS}V_p^*=E_{pqS}$;
\item $E_{pS}E_{qS}=E_{pS\cap qS}$
\end{enumerate}
for all $p,q\in S$.
\end{lemma}

\noindent Recall from \cite{Li1} that for each right ideal $X$ and $p\in S$, the sets
\[pX = \{px \mid x \in X\}\quad \text{and}\quad p^{-1} X = \{y\in S\mid py \in X\}\]
are also right ideals. Li \cite[\S2.1]{Li1} defines the set of {\em
constructible right ideals} $\JJ(S)$ to be the
smallest family of right ideals of $S$ satisfying
\begin{enumerate}
 \item[(1)] $\emptyset, S \in \JJ(S)$ and
 \item[(2)] $X \in \JJ(S),p \in S\Longrightarrow pX,p^{-1} X \in
 \JJ(S)$.
\end{enumerate}
An inductive argument as in the proof  of \cite[Lemma 3.3]{Li1} shows that (1) and (2) imply
\begin{enumerate}
 \item[(3)] $X, Y \in \JJ(S)\Longrightarrow X \cap Y \in \JJ(S)$.
\end{enumerate}

The full $C^*$-algebra for a left cancellative semigroup $S$ will be defined
in terms of generators and relations similar to what is done in \cite{Li1} for
semigroups with identity.

\begin{definition}\label{def: Li's full algebra}
Let $S$ be a left cancellative semigroup. The {\em full semigroup
$C^*$-algebra} $C^*(S)$ is the universal unital $C^*$-algebra generated by
isometries
$(v_p)_{p\in S}$ and projections\\
$(e_X)_{X\in\JJ(S)}$ satisfying
\begin{enumerate}
\item[(L1)] $v_pv_q=v_{pq}$;
\item[(L2)] $v_pe_Xv_p^*=e_{pX}$;
\item[(L3)] $e_\emptyset=0$ and $e_S=1$; and
\item[(L4)] $e_Xe_Y=e_{X\cap Y}$,
\end{enumerate}
for all $p,q\in S$, $X,Y\in\JJ(S)$.
\end{definition}

\noindent The \emph{left regular representation} is, by definition, the $*$-homomorphism  $\lambda:C^*(S)\to C_r^*(S)$ given by
$\lambda(v_p)=V_p$ for all $p\in S$.

In \cite{Li1}, the set of constructible right ideals $\JJ(S)$ is called {\em independent} if
for every choice of $X,X_1,\dots,X_n\in\JJ(S)$ we have
\[
X_j\varsubsetneqq X\text{ for all $1\le j\le n$}\Longrightarrow \bigcup_{j=1}^n
X_j\varsubsetneqq X.
\]
Equivalently, $\JJ(S)$ is independent if $\cup_{j=1}^n X_j= X$ implies $X_j=X$ for some $1\le j\le n$.

The next two lemmas explain why right LCM semigroups form a particularly tractable
class of semigroups. The proof of the first of these lemmas is left to the reader.

\begin{lemma}\label{lem: constr right ideals for right LCM sgp}
If $S$ is a right LCM semigroup, then $\JJ(S) = \{\emptyset, S\}\cup
\{pS\mid p \in S\}$.
\end{lemma}

\begin{lemma}\label{lem: independence for granted}
Let $S$ be a right LCM semigroup. Then $\bigcup_{X \in F}{X}
\subsetneqq S$ holds for all finite subsets $F \subset \JJ(S) \setminus \{S\}$ if and
only if $\JJ(S)$ is independent.
\end{lemma}
\begin{proof}
Clearly, independence of $\JJ(S)$ implies $\bigcup_{X \in F}{X} \subsetneqq S$ for all finite $F \subset \JJ(S) \setminus \{S\}$. Conversely, let $X,X_{1},\dots,X_{n} \in \JJ(S)$ satisfy $X_i \subsetneqq X$. Since $S$ is right LCM, Lemma~\ref{lem: constr right ideals for right LCM sgp} gives $p,p_1,\dots,p_n \in S$ with $X = pS, X_i = p_iS$ for $i = 1,\dots,n$. For each $i=1,\dots, n$, $X_i \subsetneqq X$ implies that $p_i = pp_{i}'$ for some $p_{i}' \in S$ with $p_i'S\subsetneqq S$. Thus
\[\bigcup\limits_{1 \leq i \leq n}{X_{i}} = p\bigcup\limits_{1 \leq i \leq n}{p_{i}'S} \text{ and } X = pS.\]
By left cancellation, $\bigcup_{1 \leq i \leq n}{X_{i}} = X$ is equivalent to $\bigcup_{1 \leq i \leq n}{p_{i}'S} = S$. However, the second statement is false by the choice of $p_{i}'S$. Hence $\bigcup_{1 \leq i \leq n}{X_{i}} \subsetneqq X$ and $\JJ(S)$ is independent.
\end{proof}

\begin{remark}\label{rem: Q^e_F,emptyset non-zero}
Let $S$ be a left cancellative semigroup and $\JJ(S)$ the family of
constructible right ideals. Let $F$ be a finite subset of $\JJ(S) \setminus
\{S\}$. Note that if $S$ has an identity $1_S$, then $\bigcup\limits_{X \in
F}{X} \subsetneqq S$ holds. Indeed, if we had $\bigcup\limits_{X \in F}{X} =
S$, then there would exist $X \in F$ such that $1_S \in X$, so $X=S$ since $X$ is a right ideal, a
contradiction.
\end{remark}

\begin{cor}\label{cor: independence for monoids}
If $S$ is a right LCM semigroup with identity, then $\JJ(S)$ is independent.
\end{cor}
\begin{proof}
This follows from \cite[Proposition 2.3.5]{No0}. Alternatively, apply Lemma~\ref{lem: independence for granted} and  Remark \ref{rem: Q^e_F,emptyset
non-zero}.
\end{proof}

\noindent If $S$ does not have an identity, we can always pass to its unitisation $\tilde{S} = S \cup \{1_S\}$, where we declare $1_S p = p = p 1_S$ for all $p \in \tilde{S}$.

\begin{lemma}\label{lem:right LCM for unitisation}
If $S$ is a right LCM semigroup with $S^*= \emptyset$, then for every $p,q\in S$ we have $pS \cap qS=\emptyset$ precisely when  $p\tilde{S} \cap q\tilde{S}=\emptyset$, and
\[pS \cap qS = rS \text{ if and only if } p\tilde{S} \cap q\tilde{S} = r\tilde{S}\]
for  $r \in S$. In particular, $\tilde{S}$ is right LCM and $\JJ(\tilde{S})$ is independent.
\end{lemma}
\begin{proof}
Let $p,q \in S$. It is clear that $pS \cap qS$ is empty if and only if $p\tilde{S} \cap q\tilde{S}$ is. Suppose next that $pS \cap qS \neq \emptyset$.  In case $pS=qS$, the standing assumption imposed on semigroups without identity element  forces $p=q$, and so $p\tilde{S} = q\tilde{S}$. Assume therefore that $pS\neq qS$, and let $r\in S$ with $pS \cap qS = rS$. Then
\[pS \cap qS = rS \subset r\tilde{S} \subseteq p\tilde{S} \cap q\tilde{S}.\]
We claim that $p\tilde{S} \cap q\tilde{S} \subseteq r\tilde{S}$.   Let $t\in p\tilde{S} \cap q\tilde{S}$. If $t\in pS\cap qS$ then clearly $t\in rS\subset r\tilde{S}$. Assume that $t=q=ps$ for some $s\in S$. Then $t=q1_S\in q\tilde{S}$ and $t\in pS\subset p\tilde{S}$, so $t\in r\tilde{S}$. The case that $t=p=qu$ for some $u\in S$ is similar, and the claim is established.

Since left cancellation in $\tilde{S}$ is inherited from $S$, this shows that $\tilde{S}$ is a right LCM semigroup  Thus $\JJ(\tilde{S})$ is independent according to Corollary~\ref{cor: independence for monoids}.
\end{proof}

\noindent The following example shows that independence of $\JJ(S)$ need not hold in general for semigroups without an identity:

\begin{example}
Let $S = 2\N^\times \cup 3\N^\times$ be endowed with composition given by
multiplication. Then $\JJ(S)$ is not independent. Indeed, for
$X_1 = 2\N^\times = 3^{-1}(2S)$ and $X_2 = 3\N^\times = 2^{-1}(3S)$, we have
$X_i \subsetneqq S$ but $X_1 \cup X_2 = S$. We remark that $S$ is not right LCM.
\end{example}

\noindent One can modify the previous example to get a  right LCM semigroup with
$S^{*} = \emptyset$ such that $\JJ(S)$ is independent.

\begin{example}\label{ex:rightLCM no units}
Consider the set $S = \N^\times \setminus \{1\}$ with composition given by
multiplication. Then $S$ is a right LCM semigroup with $S^{*} = \emptyset$. We claim that $\JJ(S)$ is independent.
For this it suffices to show that $\bigcup_{X \in F}{X}
\subsetneqq S$ holds for all finite $F \subset \JJ(S) \setminus \{S\}$. Assume that $\bigcup_{i=1}^n{X_i}=S$ for $X_1,\dots ,X_n$ in $\JJ(S) \setminus \{S\}$. Since $S$ contains $n+1$ relatively prime elements $p_1, \dots ,p_{n+1}$, we can find $i_0\in \{1,\dots ,n\}$ and $j,k\in \{1,\dots, n+1\}$ with $j\neq k$ such that $p_j, p_k\in X_{i_0}$. But this implies that $X_{i_0}=S$, a contradiction. The underlying idea is that as long as there are infinitely many prime right ideals,
$\JJ(S)$ is independent.
\end{example}

\begin{remark}\label{rem: range proj of the gen isometries}
For a left cancellative semigroup $S$, the range projection $v_pv_p^*$ of the
generating isometry $v_p$ in $C^*(S)$ equals $e_{pS}$:
\[v_pv_p^* \stackrel{(L3)}{=} v_pe_Sv_p^* \stackrel{(L2)}{=} e_{pS}.\]
Thus, if $S$ has an identity, then $v_x$ is a unitary in $C^*(S)$ if (and only
if) $x \in S^*$. If $S$ is right LCM, then Lemma~\ref{lem: constr right ideals
for right LCM sgp} shows that $C^*(S)$ is generated already by $(v_p)_{p \in
S}$.\\
\end{remark}

\subsection{Spanning families and distinguished subalgebras.}~\\
When $S$ is a right LCM semigroup we have a  description of its
$C^*$-algebra $C^*(S)$ in terms of a spanning set of monomials of the kind that
span $C^*$-algebras associated to quasi-lattice ordered pairs, see
\cite{LaRa}. This assertion could be deduced from \cite[Proposition
3.2.15]{No0}, however we include a proof since we here do not assume that $S$
necessarily has an identity.

\begin{lemma}\label{lem: spanning family for algebra}
Let $S$ be a right LCM semigroup. If $S$ has an identity, then $C^*(S)=\clsp\{v_pv_q^* \mid p,q\in
S\}$. If $S^* = \emptyset$, then $C^*(S)=\clsp\{v_pv_q^* \mid p,q\in \tilde{S}\}$.
\end{lemma}
\begin{proof}
In each case, the right-hand side is closed under taking adjoints and, due to
Remark~\ref{rem: range proj of the gen isometries}, contains the generators of
$C^*(S)$. Hence, we only need to show that the right-hand side is
multiplicatively closed. Using (L1), it suffices to show that the product of
$v_q^*$ and $v_p$ for arbitrary $p$ and $q$ in $S$ is $0$ or has the form
$v_{p'}v_{q'}^*$ for some $p',q' \in S$. By Remark~\ref{rem: range proj of the
gen isometries}, we have
\[v_q^*v_p = v_q^*e_{qS}e_{pS}v_p \stackrel{(L4)}{=} v_q^*e_{qS \cap pS}v_p.\]
Since $S$ is right LCM, we know that $pS \cap qS$ is either empty, in which
case $e_{qS \cap pS} = 0$ by (L3), or $pS \cap qS = rS$ for some $r \in pS
\cap qS$. If we let $p',q'\in S$ be such that $pp' = qq'=r$ in $S$ (which are
uniquely determined since $S$ is left cancellative), then
\[v_q^*v_p=v_q^*e_{rS}v_p=v_q^*v_{qq'}v_{pp'}^*v_p=v_{q'}v_{p'}^*\]
establishes the claim for the second case.
\end{proof}

\begin{definition} Let $S$ be a left cancellative semigroup. Define a
subalgebra of $C^*(S)$ by
\begin{equation*}
\DD:=C^*(\{e_{X}\mid X\in \JJ(S)\}).
\end{equation*}
If $S^*\neq \emptyset$, define further the subalgebras
\[
\CC_O:=C^*(\{v_pv_xv_p^*\mid p\in S,x\in S^*\})\text{ and }
\CC_I:=C^*(\{e_{pS},v_x\mid p\in S,x\in S^*\}).
\]
These are, respectively, the {\em diagonal}, the {\em outer core} and the {\em
inner core} of $C^*(S)$.
\end{definition}

\noindent It is clear that $\DD=\clsp\{e_X \mid X\in \JJ(S)\}$. The other two
subalgebras satisfy the following:

\begin{lemma}\label{lem: properties of the subalgebras}
Let $S$ be a right LCM semigroup with $S^*\neq \emptyset$. Then
\begin{enumerate}
\item[(i)] $\DD\subseteq\CC_I\subseteq\CC_O$;
\item[(ii)] $\CC_I=\clsp\{e_{pS}v_x \mid p\in S,~x\in S^* \}$; and
\item[(iii)] if $S^*=\{1_S\}$, then $\DD=\CC_I=\CC_O$. %old (ii)
\end{enumerate}
\end{lemma}
\begin{proof}
Parts (i) and (iii) are immediate verifications. For assertion (ii) we use
(L2) and (L4) to get
\[
e_{pS}v_xe_{qS}v_y=e_{pS}v_xe_{qS}v_x^*v_xv_y=e_{pS\cap
xqS}v_{xy},
\]
for each $p,q\in S$, $x,y\in S^*$.
Hence $\{e_{pS}v_x \mid p\in S,~x\in S^*\}$ is closed under
multiplication. Since $(e_{pS}v_x)^*=
v_x^*e_{pS}=e_{x^{-1}pS}v_{x^{-1}}$, claim (ii) follows.
\end{proof}

\subsection{\texorpdfstring{Conditional expectations onto canonical diagonals of $C^*(S)$ and $C_r^*(S)$}{Conditional expectations onto canonical diagonals of full and reduced semigroup C*-algebra}}\label{subsection:cond-exp-diagonal}~\\
 Let $S$ be a left cancellative semigroup. The diagonal $\DD_r$ in $C_r^*(S)$
 is defined to be the subalgebra $\DD_r=\clsp\{E_X \mid X\in \JJ(S)\}$. We
 show next that when $S$ is right LCM and also {right cancellative},
 there is a canonical faithful conditional expectation from  $C_r^*(S)$ onto
 its diagonal. The result was motivated by  \cite[Lemma 3.11]{Li1}, and is a
 generalisation to cancellative right LCM semigroups of a similar result
 proved for quasi-lattice ordered groups, see \cite[Remark 3.6]{Ni} and
 \cite{QuiRa}. More precisely, it is a consequence of the normality of the
 coaction in \cite[Proposition 6.5]{QuiRa}  and of \cite[Lemma 6.7]{QuiRa}
 that the Wiener-Hopf algebra $\mathcal{T}(G, S)$, i.e. the reduced
 $C^*$-algebra of a quasi-lattice ordered group $(G, S)$, admits a  faithful
 conditional expectation onto its canonical diagonal.

\begin{prop}\label{prop: faithful cond exp for the red algebra}
If $S$ is a cancellative right LCM semigroup, then the canonical map
$\Phi_{\DD,r}:C^*_r(S) \longrightarrow \DD_r$ given by $\Phi_{\DD,
r}(V_pV_q^*)=\delta_{p,q}V_pV_p^*$ for $p,q\in S$ is a faithful conditional
expectation.
\end{prop}
\begin{proof}
It was proved in \cite[Section 3.2]{Li1} that there is a faithful conditional
expectation $E:\mathcal{L}(\ell^2(S)) \longrightarrow \ell^\infty(S)$ characterised by
$\langle E(T)\varepsilon_s,\varepsilon_s\rangle = \langle T\varepsilon_s,\varepsilon_s\rangle$ for all $s \in
S$ and all $T \in \mathcal{L}(\ell^2(S))$. Clearly, $\DD_r \subset \ell^\infty(S)$. We
will show that the converse inclusion holds. Note that $C^*_r(S)$ is the closure
of the span of elements $V_pV_q^*, p,q \in S$. Therefore it suffices to show that
$E(V_pV_q^*) \in \DD_r$ for any $p, q\in S$. Let $s \in S$. If  $s \notin qS$,
then $V_q^*\varepsilon_s = 0$, and for $s\in qS$ of the form $s=qs'$ we have
$V_q^*\varepsilon_s = \varepsilon_s'$. Thus if $E(V_pV_q^*) \neq 0$, then there is $s' \in S$
such that $ps' = qs'$. Right cancellation then implies $p=q$, so $V_pV_q^* \in
\DD_r$. Since $\Phi_{\DD,r} = E$ in this case, the proposition follows.
\end{proof}

\noindent A successful strategy to prove injectivity of representations of $C^*(S)$
uses the classical idea of Cuntz from \cite{Cun}, which involves
expecting onto a diagonal subalgebra and constructing a projection with good
approximation properties. To pursue this path, we need a  faithful conditional
expectation from $C^*(S)$ onto  $\DD$. Such a map can be specified by its image on the spanning elements of $C^*(S)$
as follows:
\begin{equation}\label{def:Phi-D}
\Phi_{\DD} (v_pv_q^*)=\begin{cases}v_pv_p^*,&\text{ if }p=q\\
0,&\text{ if }p\neq q.\end{cases}
\end{equation}
Thus in examples we need to ensure that \eqref{def:Phi-D} does extend to
$C^*(S)$ and that it is faithful on positive elements. We now describe one
such situation.

Let us recall the notion of a semigroup crossed product by  endomorphisms, see
e.g. \cite{LaRa}. Let $S$ be a semigroup with identity and $A$ a unital
$C^*$-algebra with an action $S \stackrel{\alpha}{\curvearrowright}A$ by
endomorphisms. A nondegenerate representation of $(A,S,\alpha)$ in a  unital
$C^*$-algebra $B$ is given by a unital $*$-homomorphism $\pi_A: A
\longrightarrow B$ and a semigroup homomorphism $\pi_S: S \longrightarrow
\text{Isom}(B)$, where $\text{Isom}(B)$ denotes the semigroup of isometries in
the $C^*$-algebra $B$. The pair $(\pi_A,\pi_S)$ is said to be covariant if it
satisfies the covariance condition
\[
\pi_S(s)\pi_A(a)\pi_S(s)^* = \pi_A(\alpha_s(a)) \text{ for all } a \in A
\text{ and }s \in S.
\]
Assuming that there is a covariant pair, the semigroup crossed product $A
\rtimes_\alpha S$ is the unital $C^*$-algebra generated by a pair $(\iota_A,
\iota_S)$ which is universal for nondegenerate covariant representations. This
is to say that whenever $(\pi_A, \pi_S)$ is a nondegenerate covariant
representation of $(A,S,\alpha)$ in a $C^*$-algebra $B$, there is a
homomorphism $\overline{\pi}: A \rtimes_\alpha S \longrightarrow B$ such that
\[
\pi_A = \overline{\pi} \circ \iota_A \text{ and } \pi_S = \overline{\pi} \circ
\iota_S.
\]
The crossed product $A \rtimes_\alpha S$ is uniquely determined (up to canonical isomorphism) by
this property. If the action $\alpha$ is by injective endomorphisms, then
there is always a covariant pair and $A \rtimes_\alpha S$ is non-trivial, see
\cite{Lac}.

%\vskip 0.3cm
It was observed in \cite{Li1} that whenever $S$ is a left cancellative semigroup with identity, then there is an action $\tau$ of $S$ by endomorphisms of $\DD$ given by $\tau_p(e_{X})=v_pe_{X}v_p^* = e_{pX}$ for all $p\in S$ and $X\in \JJ(S)$. The semigroup crossed product $\DD\rtimes_\tau S$ is the universal $C^*$-algebra generated by a pair $(\iota_\DD, \iota_S)$ of homomorphisms of $\DD$ and $S$, respectively, subject to the covariance condition $\iota_S(p)\iota_\DD(e_{X})\iota_S(p)^*=\iota_\DD(e_{pX})$ for all $p\in S$ and $X\in \JJ(S)$. As shown in \cite[Lemma 2.14]{Li1}, the $C^*$-algebras $C^*(S)$ and $\DD\rtimes_\tau S$ are canonically isomorphic, through the isomorphism that sends $v_p$ to $\iota_S(p)$ and $e_{X}$ to $\iota_\DD(e_{X})$. We have the following consequence of  Lemma~\ref{lem: spanning family for algebra}.

\begin{cor} Given a right LCM semigroup $S$, let $\tau$ be the action of $S$
on $\DD$ given by conjugation with $v_p$ for $p\in S$. If $S$ has an identity, then
$\DD\rtimes_\tau S=\clsp\{\iota_S(p)\iota_S(q)^*\mid p,q \in S\}.$ If $S^*=\emptyset$, then $\DD\rtimes_\tau S=\clsp\{\iota_S(p)\iota_S(q)^*\mid p,q \in \tilde{S}\}$ holds.
\end{cor}

\noindent Recall that a semigroup $S$ is said to be \emph{right reversible} if $Sp \cap Sq$ is non-empty for all $p,q \in S$, see \cite[§10.3]{Cli-Pre}. If $S$ embeds into a group, we refer to the subgroup generated by the image of $S$ as the \emph{enveloping group} of $S$. Note that this group is unique up to canonical isomorphism in case it exists.

\begin{prop}\label{prop:left-inverse} Let $S$ be a right LCM semigroup with identity such that $S$ is
right reversible and its enveloping group $\mathcal{G}=S^{-1}S$ is amenable.
Then there is a faithful conditional expectation from $C^*(S)$ onto $\DD$
characterised by \eqref{def:Phi-D}.
\end{prop}
\begin{proof}
 The first observation is that the action $\tau$ admits a left inverse,
 $\beta$, given by
\[\beta_p(e_X)=v_p^*e_{X}v_p = e_{p^{-1}X}\]
for $p\in S$ and $X \in \JJ(S)$. It was proved in \cite[Corollary 2.9]{Li2}
that $\beta_p$ defines an endomorphism of $\DD$ for each $p\in S$, the reason
for this being that $p^{-1}X \cap p^{-1}Y = p^{-1}(X \cap Y)$ holds for
all $X,Y \in \JJ(S)$.  It is clear that  $\beta$ is an action of $S$ such that
$\beta_p\circ \tau_p=\operatorname{id}$ for all $p\in S$.   Moreover,
\[
(\tau_p\circ \beta_p)(e_X)=v_pv_p^*e_Xv_pv_p^*=e_{pS}e_Xe_{pS}=e_X\tau_p(1)
\]
for every $p\in S$ and $X\in \JJ(S)$. Thus $\tau_p\circ \beta_p$ is simply
the cut-down to the corner associated to the projection $\tau_p(1)$.

One consequence of the existence of $\beta$ is that $\tau_p$ is injective
for every $p\in S$. Hence \cite[Theorems 2.1 and 2.4]{Lac} show that $\DD$
embeds in $\DD\rtimes_\tau S$.

 As a second consequence of the existence of $\beta$, note that
 \cite[Proposition 3.1(1)]{Lar} implies that there is a coaction of
 $\mathcal{G}$ whose fixed-point algebra is $\iota_\DD(\DD)$. Thus there is a
 conditional expectation $\Phi_\DD$ from $\DD\rtimes_\tau S$ onto
 $\iota_\DD(\DD)$ such that
\[
\Phi_\DD(\iota_S(p)\iota_S(q)^*)=\begin{cases} \iota_S(p)\iota_S(p)^*,&\text{
if }p=q\\
0,&\text{ if }p\neq q.
\end{cases}
\]
Identifying $\iota_S(p)$ with $v_p$ and $\iota_\DD(e_{pS})$ with $e_{pS}$
gives existence of the claimed expectation.
Under the assumption that the enveloping group $\mathcal{G}$ is amenable, the
map $\Phi_\DD$ is faithful on positive elements, cf. \cite[Lemma 1.4]{Qui}.
Note that the last conclusion may also be reached for the semigroup dynamical
system $(\DD, S, \tau)$ by invoking  \cite[Lemma 8.2.5]{CEL1}.
\end{proof}

\subsection{\texorpdfstring{From quasi-lattice order groups to right LCM semigroups.}{From quasi-lattice ordered groups to right LCM semigroups.}}\label{subsection-LR}~\\
It turns out that a good part of the general strategy of Laca and Raeburn
\cite{LaRa} for proving injectivity of representations of $C^*(S)$ in the case
that $S$ is part of a quasi-lattice order $(G, S)$ can be extended to the class of
right LCM semigroups, although the arguments become more delicate due to the presence of non-trivial units. The next several results make this claim precise.

\begin{notation}\label{notation: projections--D} In Lemma~\ref{lem:def-V-E-reduced-alg} we introduced isometries $V_p$ for $p\in S$ and projections $E_{pS}$ for $pS\in \JJ(S)$ in $C_r^*(S)$ that satisfy conditions (L1)-(L4). Later in the paper we shall mainly be interested in families of isometries and projections satisfying (L1)-(L4) inside an arbitrary $C^*$-algebra $B$. In order to avoid unnecessary notational adornment we shall still use $V_p, E_{pS}$ in that case.

Given a family of commuting projections $(E_{i})_{i \in I}$ in a unital $C^*$-algebra $B$ and finite subsets $A \subset F$ of $I$, we denote
\[Q_{F,A}^{E} := \prod\limits_{i \in A}E_{i}\prod_{j \in F \setminus A}(1-E_{j}).\]
If  the family is $(e_X)_{X\in \JJ(S)}$ in $C^*(S)$, we write $Q^e_{F, A}$ for the corresponding projections. In the case of a right LCM semigroup $S$, finite subsets of $\JJ(S)$ are determined by finite subsets of $S$, see Lemma~\ref{lem: constr right ideals for right LCM sgp}.
\end{notation}

\noindent If $S$ is a left cancellative semigroup with identity such that $\JJ(S)$ is independent, then \cite[Corollary 2.22]{Li1} and \cite[Proposition
2.24]{Li1} show that the left regular representation $\lambda$ from $C^*(S)$ to $C_r^*(S)$
restricts to an isomorphism from $\DD$ onto the diagonal $\DD_r$. This allows us to show:

\begin{lemma}\label{lem:lambda-iso-diagonal-general}
Let $S$ be a right LCM semigroup. Then the left regular representation $\lambda$ restricts to an isomorphism from the diagonal $\DD$ of $C^*(S)$ onto the diagonal $\DD_r$ of $C_r^*(S)$.
\end{lemma}
\begin{proof}
If $S$ has an identity, then $\JJ(S)$ is independent by Corollary~\ref{cor: independence for monoids}. Hence the lemma is  simply  an application of the mentioned results from \cite{Li1}. Now suppose $S^*= \emptyset$ holds. Then $\JJ(\tilde{S})$ is independent according to Lemma~\ref{lem:right LCM for unitisation} and Corollary~\ref{cor: independence for monoids}. Moreover, by Lemma~\ref{lem:right LCM for unitisation}, we have
\[pS \cap qS = rS \text{ if and only if } p\tilde{S} \cap q\tilde{S} = r\tilde{S} \text{ for all } p,q,r \in S.
\]
This fact and the standing hypothesis $S \neq \emptyset$ imply that the maps
\[\begin{array}{lclclcl}
\DD &\longrightarrow& \tilde{\DD} &\hspace*{4mm}\text{and}\hspace*{4mm}& \DD_r &\longrightarrow& \tilde{\DD}_r\\
e_S &\mapsto& e_{\tilde{S}}&&E_S &\mapsto& E_{\tilde{S}}\\
e_{pS} &\mapsto& e_{p\tilde{S}}&&E_{pS} &\mapsto& E_{p\tilde{S}}
\end{array}\]
are isomorphisms, where $\tilde{\DD}$ and $\tilde{\DD}_r$ denote the diagonal subalgebra of $C^*(\tilde{S})$ and $C^*_r(\tilde{S})$, respectively. Since $\JJ(\tilde{S})$ is independent, $\tilde{\lambda}:\tilde{\DD} \longrightarrow \tilde{\DD}_r$ is an isomorphism. Altogether, we get a commutative diagram
\begin{equation}\label{com:diagram DD unitise S}
\begin{tikzpicture}[>=stealth, xscale=2, yscale=1.5]
\node (c) at (0,-2) {$\tilde{\DD}$};
\node (b) at (2,0) {$\DD_r$};
\node (a) at (0,0) {$\DD$}
edge[->](b)
edge[->] (c);
\node (d) at (2,-2) {$\tilde{\DD}_r$}
edge[<-] (b)
edge[<-] (c);
\node at (1,0.2) {$\lambda|_{\DD}$};
\node at (-0.2,-1) {$\cong$};
\node at (2.2,-1) {$\cong$};
\node at (1,-2.2) {$\cong$};
\end{tikzpicture}
\end{equation}
which proves that $\lambda|_{\DD}$ is an isomorphism.
\end{proof}

\begin{prop}\label{prop: part (i) independent}
Suppose $S$ is a right LCM semigroup and $\pi$ is a $*$-homomorphism of
$C^*(S)$. Let $E_X:=\pi(e_X)$ for $X\in \JJ(S)$ and
$V_p:=\pi(v_p)$ for $p\in S$. Then the following statements are equivalent:
\begin{enumerate}[(I)]
\item\label{it:I1}$\pi\vert_\DD:\DD\longrightarrow \pi(\DD)$ is an isomorphism.
\item\label{it:I2} $Q_{F,A}^{E} \neq 0$
for all non-empty finite subsets $F$ of $\JJ(S)$ and all non-empty subsets
$A\subset F$ satisfying
\[
\bigcap\limits_{X \in A}{X} \cap \bigcap\limits_{Y \in F \setminus A}{S
\setminus Y} \neq \varnothing.
\]
\item\label{it:I3} $Q_{F, \emptyset}^E\neq 0$ for all non-empty subsets
$F\subset \JJ(S)\setminus\{S\}$.
\end{enumerate}
\end{prop}
\begin{proof}
Lemma~\ref{lem:lambda-iso-diagonal-general} implies
that the left regular representation $\lambda$
restricts to an isomorphism from $\DD$ onto $\DD_r$. Thus assuming \emph{(I)} and letting
$A\subset F$ be finite non-empty subsets of $\JJ(S)$ satisfying the non-empty
intersection condition of \emph{(II)}, it follows that $\lambda(Q^e_{F,
A})\neq 0$. Hence $Q^e_{F, A}\neq 0$, which by injectivity of $\pi\vert_\DD$ gives
$Q^E_{F, A}\neq 0$. This shows that \emph{(I)} implies \emph{(II)}. Conversely, it suffices to note that by \cite[Lemma 2.20]{Li1},
condition $(I)$ is equivalent to the implication $Q_{F,A}^{E} = 0
\Longrightarrow Q_{F,A}^{e} = 0$ for all non-empty finite subsets $F$ of
$\JJ(S)$ and all non-empty subsets $A \subset F$. Thus \emph{(I)} and  \emph{(II)} are equivalent.

Consider next a non-empty finite subset $F\subset \JJ(S)\setminus\{S\}$. If $S$ has an identity, then Lemma~\ref{lem: independence for granted} provides independence of $\JJ(S)$. In particular, we have $\bigcup_{X \in F}{X} \subsetneqq S$. Hence $Q_{F,\emptyset}^{e} \neq 0$ because its image under $\lambda$ is non-zero. In case $S^*=\emptyset$, Lemma~\ref{lem:right LCM for unitisation} shows that $F \subset \JJ(S)\setminus\{S\}$ corresponds to a finite subset $\tilde{F} \subset \JJ(\tilde{S})\setminus\{\tilde{S}\}$. As $\tilde{S}$ is a right LCM semigroup with identity, we get $Q_{\tilde{F},\emptyset}^{e} \neq 0$. According to Lemma~\ref{lem:lambda-iso-diagonal-general}, this is equivalent to $Q_{F,\emptyset}^{e} \neq 0$. Since $\pi$ carries $Q_{F,\emptyset}^e$ to $Q_{F, \emptyset}^E$, it follows that \emph(I) implies \emph{(III)}.

Thus it remains to prove that \emph{(III)} yields \emph{(II)}. Assume \emph{(III)} and let $F\subset\JJ(S)$ be a non-empty subset and $A\subset F$ non-empty satisfying the non-empty intersection condition of \emph{(II)}. Let $\sigma_A \in S$ such that $\sigma_AS = \bigcap\limits_{X \in A}{X}$ and
$\bigcup\limits_{Y \in F \setminus A}{Y} \neq S$. Thus,
\[Q_{F,A}^E = Q_{A,A}^E\,Q_{F \setminus A,\emptyset}^E\,Q_{A,A}^E =
V_{\sigma_A}\prod\limits_{Y \in F \setminus
A}{(1-V_{\sigma_A}^*E_YV_{\sigma_A})}V_{\sigma_A}^*.
\]
Each $Y\in F\setminus A$ has the form
$Y=p_{Y}S$ for some $p_Y\in S$. Since $S$ is right LCM, there exists $q_Y\in S$ such that
$\sigma_A^{-1}Y=q_YS$ and $\sigma_A q_YS=\sigma_AS\cap p_YS$.
Thus $\sigma_A^{-1}Y$ is a proper
right ideal of $S$ if and only if $\sigma_A \notin Y$. The choice of $F$ and
$A$ therefore guarantees that $\sigma_A^{-1}Y \neq S$ for all $Y \in F
\setminus A$. Hence $Q_{\sigma_A^{-1}(F\setminus A), \emptyset}^E\neq 0$ by
\emph{(III)}.  From
\[
Q_{\sigma_A^{-1}(F\setminus A), \emptyset}^E=\prod\limits_{Y \in F \setminus
A}{(1-E_{\sigma_A^{-1}(Y)})}
\]
and $V_{\sigma_A}^*E_YV_{\sigma_A} = E_{\sigma_A^{-1}(Y)}$, we obtain that
\[
Q_{F,A}^E = V_{\sigma_A}\prod\limits_{Y \in F \setminus
A}{(1-E_{\sigma_A^{-1}(Y)})}V_{\sigma_A}^*  \neq 0
\]
since $V_{\sigma_A}$ is an isometry. This finishes the proof of the proposition.
\end{proof}

\noindent The following result is a variant of  \cite[Lemma 1.4]{LaRa}.

\begin{lemma}\label{lem: first Qs--D} If $(E_{i})_{I}$ are commuting
projections in a unital $C^*$-algebra
$B$ and $A \subset F$ are finite subsets of $I$, then
each $Q_{F,A}^{E}$ is a projection, $\sum\limits_{A \subset F}Q_{F,A}^{E} = 1$, we have
\begin{equation}\label{eq: 1 for the first Qs--D}
\sum_{i \in F}\lambda_{i}E_{i}=\sum_{A\subset F}\Big(\sum_{i\in
A}\lambda_{i}\Big) Q_{F,A}^{E}
\end{equation}
for any choice of complex numbers $\{\lambda_i \mid i\in I\}$ and, moreover,
\begin{equation}\label{eq: 2 for the first Qs--D}
\Big\|\sum_{i\in F}\lambda_{i}E_{i}\Big\|=\max\limits_{\substack{A \subset F
\\ Q_{F,A}^{E} \neq 0}}\big|\sum_{i\in A}\lambda_{i}\big|.
\end{equation}
\end{lemma}
\begin{proof}
 Since the projections $E_{i}$ commute, $Q_{F,A}^{E}$ is a projection. The
 second assertion is obtained via \[1 = \prod\limits_{i \in F}{(E_{i} + 1-
 E_{i})}  = \sum\limits_{A \subset F}Q_{F,A}^{E}.\] Equation~\eqref{eq: 1 for
 the first Qs--D} as well as Equation~\eqref{eq: 2 for the first Qs--D} follow
 immediately from this.
\end{proof}

\noindent We now set up a conventional notation which will be used repeatedly in the sequel. Let $S$ be a right LCM semigroup. We let
\begin{equation}\label{eq:tF-tFD}
t_F := \sum\limits_{p,q \in F} \lambda_{p,q}v_pv_q^* \,\,\text{ and }\,\,
t_{F,\DD} := \sum\limits_{p \in F} \lambda_{p,p}e_{pS}
\end{equation}
denote an arbitrary, but fixed  finite linear combination in $C^*(S)$ and its
image in $\DD$ under $\Phi_\DD$, where $F$ is a finite subset of $S$ when $S$ has an identity, or, in case $S^*=\emptyset$, $F$ is a finite subset of $\tilde{S}$, and $\lambda_{p,q}\in \C$  for $p,q \in F$.

We will decompose $t_F - t_{F,\DD}$ into further terms,
based on a suitable subset $A \subset F$ depending on the choice of the
$\lambda_{p,q}$'s. We are interested in combinations $t_F$ with $t_{F, \DD}\neq
0$, so we shall make this a standing assumption.

\begin{lemma}\label{lem: basic projection--D}
Let $S$ be a right LCM semigroup and $t_F$, $t_{F,\DD}$ be as in
\eqref{eq:tF-tFD}. Then there exists a non-empty subset $A\subset F$ such that
the projection
$Q_{F,A}^e$ is non-zero and satisfies the following:
\begin{enumerate}
\item[(i)] $Q_{F,A}^ev_pv_q^*Q_{F,A}^e=0$ for all $p,q\in F$ with $p\not\in A$
or $q\not\in A$.
\item[(ii)] $\Vert Q_{F,A}^{e} t_{F,\DD}Q_{F,A}^{e}\Vert= \|t_{F,\DD}\|$.
\item[(iii)] If $t_{F, \DD}$ is positive, then we may take
$Q_{F,A}^{e}t_{F,\DD}Q_{F,A}^{e}=\|t_{F,\DD}\|Q_{F,A}^{e} $.
\end{enumerate}
\end{lemma}
\begin{proof}
The projections $(e_{pS})_{p \in F}$ commute because of $e_{pS}e_{qS} = e_{pS
\cap qS}$ for any  $p, q\in S$. Applying Lemma \ref{lem: first Qs--D} yields
$A \subset F$ which satisfies $Q_{F,A}^e\not=0$, and
\[\Vert Q_{F,A}^{e} t_{F,\DD}Q_{F,A}^{e}\Vert =\|t_{F,\DD}\|.\]
If $t_F$ is positive, then  we may choose $Q_{F,A}^{e}t_{F,\DD}Q_{F,A}^{e}$ to
be a multiple of $Q_{F,A}^{e}$. As $t_{F,\DD} \neq 0$, we must have $A \neq
\emptyset$. The fact that $Q_{F,A}^e \neq 0$ and the right LCM property of $S$
imply that
\[
Q_{F,A}^e = \prod\limits_{q\in F\setminus A}(e_{\sigma_AS}-e_{\sigma_AS\cap
qS}),
 \]
where $\sigma_A \in S$ is such that $\sigma_AS=\cap_{p\in A}pS$.  We claim that $Q_{F,A}^{e}v_pv_q^*Q_{F,A}^{e}=0$ for $p \in F \setminus A$. Indeed, if we have $p\not\in A$, then $Q_{F,A}^{e}v_pv_q^*Q_{F,A}^{e}$ contains a factor of $(1-e_{pS})v_p=v_p-v_p=0$, and hence $Q_{F,A}^{e}v_pv_q^*Q_{F,A}^{e}=0$. Similarly, $v_q^*(1-e_{qS}) = 0$, so we get $Q_{F,A}^{e}v_pv_q^*Q_{F,A}^{e}=0$ for $q \in F \setminus A$.
\end{proof}

\noindent Before we state the next result we introduce some notation. Assume the
hypotheses of Lemma~\ref{lem: basic projection--D} and let $A$ be the finite
subset of $F$ satisfying (i)-(iii). Fix $\sigma_A \in S$ such that
$\bigcap_{p \in A}pS = \sigma_AS$ (this element is not unique for the
given $A$; in case $S^*\neq \emptyset$ then $\sigma_Ax$ for any $x \in S^*$
will satisfy the same identity as $\sigma_A$). For each $p\in A$, let $p_A \in
S$ denote the element satisfying $pp_A = \sigma_A$. By left cancellation, this
element is unique. Define now \[\begin{array}{lcl}
t_{F,1}
&=&\sum\limits_{\substack{p,q\in F, p\neq q\\p\notin A\text{ or } q\notin A}}
\lambda_{p,q}v_pv_q^*, \label{def:TF1--D}\vspace*{2mm}\\
t_{F,2}
&=& \sum\limits_{\substack{p,q\in A, p\neq q\\p_A S\neq q_A S}}
\lambda_{p,q}v_pv_q^*\vspace*{2mm},\text{ and}\\
t_{F,3}
&=& \sum\limits_{\substack{p,q\in A, p\neq q\\p_A S=q_A S}}
\lambda_{p,q}v_pv_q^*.
\end{array}\]

The sum $t_{F,3}$ will only be relevant here when $|S^*|>1$. When $|S^*|\leq
1$, we distinguish two cases: if
$S^*=\emptyset$, our standing assumption says that
$sS=tS$ forces $s=t$ for $s, t\in S$. Hence a term in $t_{F, 3}$ would
correspond to $p_A=q_A$, which implies $pp_A=qp_A$. Thus, if the semigroup $S$
is also right cancellative, we would get $p=q$, a contradiction.  The same
argument rules out $t_{F,3}$ when $S^*=\{1_S\}$.

\begin{lemma}\label{lem: decomposition of t_F--D}
 Assume the hypotheses of Lemma~\ref{lem: basic projection--D} and let $A$ be
 the finite subset of $F$ satisfying (i)-(iii).  Fix $\sigma_A$ as above.
 Define a subset of $A\times A$ by
\[
A_1 = \{(p,q)\mid p\neq q, \exists x,y\in S, x,y \text { not both units}:
p_Ax=q_Ay, p_AS \cap q_AS = p_AxS\}.
\]
Then
\begin{align}
e_{\sigma_AS}t_{F,2}e_{\sigma_AS}
&= \sum\limits_{(p,q) \in A_1}
\lambda_{p,q}v_{\sigma_Ax}v_{\sigma_Ay}^*.\label{def:TF2Q1--D}
\end{align}
If $|S^*|>1$, then also
\begin{align}
e_{\sigma_AS}t_{F,3}e_{\sigma_AS}&= \sum\limits_{\substack{p,q\in A, p\neq
q,\\ p_A = q_Ax, x\in S^*}}
\lambda_{p,q}v_{\sigma_A}v_{x}v_{\sigma_A}^*.\label{def:TF3Q1--D}
\end{align}
\end{lemma}
\begin{proof}
Clearly,  $t_F= t_{F,\DD}+t_{F,1}+t_{F,2}+t_{F,3}$.
Let us look more closely at the cut-downs of $t_{F,2}$ and $t_{F,3}$ by
$Q_{F,A}^e$. For $p,q \in A, p \neq q$, we have
\[\begin{array}{lcl}
e_{\sigma_AS}v_pv_q^*e_{\sigma_AS} &=&
v_{\sigma_A}v_{p_A}^*v_{q_A}v_{\sigma_A}^* \\
&=& v_{\sigma_A}v_{p_A}^*e_{p_AS}e_{q_AS}v_{q_A}v_{\sigma_A}^*\vspace*{2mm}\\
&=& \begin{cases} 0,&\text{ if } p_AS \cap q_AS = \emptyset\\
v_{\sigma_Ax}v_{\sigma_Ay}^*,&\text{ if } p_AS \cap q_AS \neq \emptyset,
\end{cases}\end{array} \]
where $x,y \in S$ satisfy $p_Ax = q_Ay$ and $p_AS \cap q_AS = p_AxS$. The
choice of the pair $(x,y)$ is unique up to composition from the right by
$S^*$. Hence, $v_{\sigma_Ax}v_{\sigma_Ay}^*$ is independent of the choice of
$(x,y)$. Therefore, with regard to $t_{F,2}$, we only have to deal with $p,q
\in A, p \neq q$ such that $p_AS \cap q_AS \neq \emptyset$. These are exactly
the pairs $(p,q)$ in $A_1$, so \eqref{def:TF2Q1--D} follows.

If $v_pv_q^*$ are terms in $t_{F,3}$ then $p_AS=q_AS$ means that
$p_A\mathcal{R}q_A$, where $\mathcal{R}$ is the right Green relation. Thus
there exists $x\in S^*$ such that $p_A=q_Ax$, and  \eqref{def:TF3Q1--D}
follows.
\end{proof}

\begin{lemma}\label{lem: second projection--D}
If $S$ is a right LCM semigroup, then there
are finite subsets $A,F_1$ of $S$ with $A \subset F \subset F_1$ and
$Q_{F_1,A}^e \neq 0$ such that
\begin{align*}
Q_{F_1,A}^e t_F Q_{F_1,A}^e &= Q_{F_1,A}^e(t_{F,\DD} + t_{F,3})Q_{F_1,A}^e
\text{ and }\\
\Vert Q_{F_1,A}^e t_{F,\DD} Q_{F_1,A}^e \Vert &= \|t_{F,\DD}\|.
\end{align*}
If $t_{F, \DD}$ is positive, then we may take $Q_{F_1,A}^e t_{F,\DD}
Q_{F_1,A}^e=\|t_{F,\DD}\| Q_{F_1,A}^e$.
\end{lemma}
\begin{proof} We invoke the notation of Lemma~\ref{lem: decomposition of
t_F--D}.
For each $(p,q)\in A_1$, let $\alpha_{p,q}\in S$ be given by
\[\alpha_{p,q}:=
\begin{cases}
x & \text{if $x\in S\setminus S^*$,}\\
y & \text{if $x\in S^*$,}\\
\end{cases}\]
and set $F_1 := F \cup \{\sigma_A\alpha_{p,q}\mid (p,q) \in A_1 \}$. First of
all, let us show that $Q_{F_1,A}^e \neq 0$ holds. Due to $Q_{F,A}^{e} \neq 0$,
we know that $\sigma_AS\cap rS$ is a proper and non-empty subset of $\sigma_AS$
for each $r\in F\setminus A$. Choose for each $r\in F\setminus A$, an element
$r'\in S\setminus S^*$ such that $\sigma_AS\cap rS=\sigma_Ar'S$. It follows
that $r'S \subsetneqq S$ and $\alpha_{p,q}S \subsetneqq S$ for all $r \in
F\setminus A$
and all $(p,q) \in A_1$.

If $S$ has an identity, $\JJ(S)$ is independent by Corollary~\ref{cor: independence for monoids} and hence we get
\[\bigcup\limits_{r\in F\setminus A}r'S \cup \bigcup\limits_{(p,q)\in
A_1}\alpha_{p,q}S \subsetneqq S\]
as both index sets are finite. By taking complements and using the implication
\emph{(I)} $\Rightarrow$ \emph{(II)} from Proposition \ref{prop: part (i)
independent}, this shows that
\[\prod\limits_{r \in F \setminus A}{(1-e_{r'S})}\prod\limits_{(p,q) \in
A_1}{(1-e_{\alpha_{p,q}S})} \neq 0.\]
In the case where $S^*=\emptyset$, we get $r'\tilde{S} \subsetneqq \tilde{S}$ and $\alpha_{p,q}\tilde{S} \subsetneqq \tilde{S}$ for all $r \in
F\setminus A$ and all $(p,q) \in A_1$. By Lemma~\ref{lem:right LCM for unitisation}, $\JJ(\tilde{S})$ is independent. If we combine this with Proposition~\ref{prop: part (i) independent} and the isomorphism $\tilde{\DD} \cong \DD$ from Lemma~\ref{lem:lambda-iso-diagonal-general}, we also get
\[\prod\limits_{r \in F \setminus A}{(1-e_{r'S})}\prod\limits_{(p,q) \in
A_1}{(1-e_{\alpha_{p,q}S})} \neq 0\]
in the case $S^*=\emptyset$.

Since $v_{\sigma_A}$ is an isometry and $Q_{F_1,A}^e$ has the form
\begin{equation}\label{eq:Q1-expanded--D}
Q_{F_1,A}^e = v_{\sigma_A}\Big(\prod\limits_{r\in F\setminus
A}(1-e_{r'S})\prod\limits_{(p,q)\in
A_1}(1-e_{\alpha_{p,q}S})\Big)v_{\sigma_A}^*,
\end{equation}
it follows that $Q_{F_1,A}^e \neq 0$. Then $Q_{F_1,A}^e$ is a non-trivial
subprojection of $Q_{F,A}^e$, so
\[\Vert Q_{F_1,A}^e t_{F,\DD} Q_{F_1,A}^e \Vert= \|t_{F,\DD}\|.\]
If $t_{F,\DD} $ is positive, then we have $Q_{F_1,A}^e t_{F,\DD}
Q_{F_1,A}^e=\|t_{F,\DD}\|Q_{F_1,A}^e$. Note that
Lemma~\ref{lem: basic projection--D} implies $Q_{F_1,A}^e t_{F,1} Q_{F_1,A}^e
= 0$, and \eqref{def:TF2Q1--D} gives
\[Q_{F_1,A}^e t_{F,2} Q_{F_1,A}^e = Q_{F_1{\setminus}A,\emptyset}^e
\sum_{(p,q) \in A_1}
\lambda_{p,q}v_{\sigma_Ax}v_{\sigma_Ay}^* Q_{F_1{\setminus}A,\emptyset}^e.
\]

Now suppose $(p,q)\in A_1$ and $\alpha_{p,q}=x$. Then
$Q_{F_1{\setminus}A,\emptyset}^ev_{\sigma_Ax}v_{\sigma_Ay}^*Q_{F_1{\setminus}A,\emptyset}^e$
 contains a factor
$(1-e_{\sigma_AxS})v_{\sigma_Ax}v_{\sigma_Ay}^* = 0$ and hence
$Q_{F_1,A}^ev_pv_q^*Q_{F_1,A}^e = 0$. A similar argument gives
$Q_{F_1,A}^ev_pv_q^*Q_{F_1,A}^e = 0$ for $(p,q)\in A_1$ and $\alpha_{p,q}=y$.
Therefore, we have verified $Q_{F_1,A}^e t_{F,2} Q_{F_1,A}^e = 0$, or in other
words
\[Q_{F_1,A}^e t_F Q_{F_1,A}^e = Q_{F_1,A}^e(t_{F,\DD} + t_{F,3})Q_{F_1,A}^e.\vspace*{-6mm}\]
\vspace*{2mm}\end{proof}

\begin{lemma}\label{lem: projection for trivial S*--D}
 Let $S$ be a \emph{cancellative} right LCM semigroup with $S^*=\emptyset$.
 Then there are finite subsets $A,F_1$ of $S$ such that
 $A \subset F \subset F_1$, $Q_{F_1,A}^e \neq 0$ and
\[\Vert Q_{F_1,A}^e t_F Q_{F_1,A}^e \Vert = \|t_{F,\DD}\|.\]
If $t_F$ is positive, then we may take $Q_{F_1,A}^e t_F
Q_{F_1,A}^e=\|t_{F,\DD}\|Q_{F_1,A}^e.$
\end{lemma}
\begin{proof}
It suffices to note that the sum $t_{F, 3}$ is empty, by our remark prior to
Lemma~\ref{lem: decomposition of t_F--D}. Then the finite subsets $A,F_1$  of
$S$ given by Lemma \ref{lem: second projection--D} satisfy the claim.
\end{proof}

\section{\texorpdfstring{Uniqueness theorem for right LCM semigroup algebras using $\DD$}{Uniqueness theorem for right LCM semigroup algebras using the diagonal}}\label{sec: diagonalsection}

\noindent In this section we prove a uniqueness theorem which involves a nonvanishing
condition on
elements of the diagonal subalgebra $\DD\subset C^*(S)$. Our theorem will
apply to right LCM semigroups $S$ satisfying additional properties, including
that $S$ must be  cancellative. One of these conditions, the one we call (D2),
is rather technical. Before we state it we introduce two other conditions,
which besides being closely related to  (D2), are also likely to have more
transparent formulations in examples. Indeed,  they are often satisfied, while
condition (D2) may be harder to obtain for large classes of semigroups.

Let $S$ be a right LCM semigroup with $S^* \neq \emptyset$ and consider the
 action $S^* \curvearrowright\JJ(S)$ given by left multiplication, that is $x\cdot X=xX$ for $x\in S^*$ and $X\in J(S)$. It is standard terminology that the action $S^*{\curvearrowright}\JJ(S)$ is effective if, by definition, for every $x$ in $S^*\setminus\{1_S\}$ there is $X \in \JJ(S)$ such that $xX \neq X$. We next introduce three other properties that a semigroup can have.

\begin{definition}\label{def: conditions on S for using D}
Let $S$ be a right LCM semigroup with $S^*\neq \emptyset$.
We say that the action $S^* \curvearrowright\JJ(S)$ given by left multiplication
is {\em strongly effective} if for all $x \in S^*\setminus \{1_S\}$ and $p \in
S$, there exists $q \in pS$ such that $xqS \neq qS$.
\end{definition}
Consider further the following two conditions that $S$ can satisfy:
\begin{enumerate}
\item[(D1)] For all $x \in S^*$ and $X \in \JJ(S)$,  we have $xX \cap X \neq
\emptyset\Longrightarrow xX = X$.
\item[(D2)] If $s_0 \in S, s_1 \in s_0S$ and $F \subset S$ is a finite subset so that
$s_1S \cap \biggl(S{\setminus}\bigcup\limits_{q \in F} qS\biggr) \neq
\emptyset,$
 then, for every $x \in S^*{\setminus}\{1_S\}$, there is $s_2 \in s_1S$
 satisfying
\[s_2S \cap \biggl(S{\setminus}\bigcup\limits_{q \in F} qS\biggr) \neq
\emptyset \text{ and } s_0^{-1}s_2S \cap xs_0^{-1}s_2S = \emptyset.\]
\end{enumerate}

\begin{remark}\label{rmk:consequence-strongly-effective}
Suppose that the action $S^*{\curvearrowright}\JJ(S)$ is strongly effective and (D1) is satisfied. It is immediate to see that for every $x \in
S^*\setminus\{1_S\}$ and $p\in S$ there exists $q\in pS$ such that $xqS\cap
qS=\emptyset$. In fact, more is true. Let $s_0\in S$, $s_1\in s_0S$, and $x\in S^*\setminus\{1_S\}$. Write $s_1=s_0r$ for some $r\in S$. By the previous observation applied to $x$ and $r\in S$ there is $r'\in rS$ such that $xr'S\cap r'S=\emptyset$. If we now  let $s_2=s_0r'\in s_1 S$, we have established condition (D2) in case $F$ is the empty set. Conversely, if condition (D2) is satisfied, then by applying it with $F$ equal the empty set and $s_0=1_S$ it follows that $S^*{\curvearrowright}\JJ(S)$ is strongly effective.
\end{remark}

\noindent For convenience, we will denote the elements of $F$ by $q_1,\dots,q_{|F|}$ whenever $F \neq \emptyset$.

\begin{thm}\label{thm: using the diagonal}
Let $S$ be a cancellative right LCM semigroup such that $\Phi_{\DD}:C^*(S)\to
\DD$ is a faithful conditional expectation. Let $(V_p)_{p\in S}$ and
$(E_{pS})_{p\in S}$ be families of
isometries and projections in a $C^*$-algebra $B$ satisfying (L1)--(L4). Let
$\pi:=\pi_{V,E}$ be the associated $*$-homomorphism from $C^*(S)$ to $B$. Assume
that one of the following conditions holds:
\begin{enumerate}[(1)]
\item $|S^*| \leq 1$.
\item $|S^*|>1$ and $S$ satisfies condition (D2).
\end{enumerate}
Then $\pi:C^*(S)\to B$ is injective if and only if
\begin{equation}\label{eq: key to faithfulness independent version}
\prod_{p\in F}(1-E_{pS}) \not= 0 \text{ for every finite }F\subset S\setminus
S^*.
\end{equation}
\end{thm}

\begin{remark}
\textnormal{(a)} We observe that, for a quasi-lattice ordered pair $(G,S)$ in the sense of
\cite{Ni}, the semigroup $S$ is right LCM with $S^*=\{1_S\}$. Thus part (3) of the theorem  recovers
\cite[Theorem 3.7]{LaRa}.

\textnormal{(b)} Note that Theorem~\ref{thm: using the diagonal} does not apply to the case where $S$ is a non-trivial group ($S = \{1_S\}$ amounts to $C^*(S) \cong \C$). The reason is that $S^* = S$ directs us to part (2) of Theorem~\ref{thm: using the diagonal} and (D2) fails in the group case for $F = \emptyset$: indeed, there exists $x \in S^*\setminus\{1_S\}$, but for every $p \in S$ we get $xpS \cap pS = S \neq \emptyset$.

\textnormal{(c)}
The hypotheses of part (1) of the theorem are satisfied in the case of the semigroup from
Example~\ref{ex:rightLCM no units} because
$\Phi_\DD$ is a faithful expectation; to see this, note that
$\mathbb{N}\setminus \{1\}$ embeds in $\mathbb{Q}_+^*\setminus\{1\}$, hence in
$\mathbb{Q}_+^*$, and the latter admits a dual action on $C^*(S)$ by
\cite[Remark 3.7]{LaRa}. Semigroups satisfying condition (D2) will be described in Examples~\ref{ex: GxP sat D2 but not D3 I} and
\ref{ex: GxP sat D2 but not D3 II}.
\end{remark}

\noindent The proof of this theorem requires some preparation. Note that by
Proposition~\ref{prop: part (i) independent}, condition \eqref{eq: key to
faithfulness independent version} is equivalent to injectivity of $\pi$ on
$\DD$. Relying on this equivalence, the key step in proving Theorem~\ref{thm:
using the diagonal} is the following intermediate result:

\begin{prop}\label{prop:extend-to-contraction-using-D}
Let $S$ be a cancellative right LCM semigroup, and let $(V_p)_{p\in S}$ and
$(E_{pS})_{p\in S}$ be families of
isometries and projections in a $C^*$-algebra $B$ satisfying (L1)--(L4). Let
$\pi$ be the associated $*$-homomorphism from $C^*(S)$ to $B$. Assume that one of
the following conditions holds:
\begin{enumerate}[(1)]
\item $|S^*| \leq 1$.
\item $|S^*|>1$ and $S$ satisfies condition (D2).
\end{enumerate}
If $\pi$ is injective on $\DD$, then the map
\[\sum_{p,q\in F}\lambda_{p,q}V_pV_q^*\mapsto \sum_{p\in
F}\lambda_{p,p}V_pV_p^*,\]
where $F \subset S$ is finite and $\lambda_{p, q}\in \mathbb{C}$, is
contractive, and hence extends to a contraction $\Phi$ of $\pi(C^*(S))$ onto
$\pi(\DD)$. %Nico thinks it is a cond.exp. - but that needs to be checked.
\end{prop}

\noindent One consequence of this proposition is that when $\pi$ is the identity
homomorphism, we obtain a contractive map from $C^*(S)$ onto $\DD$ which is
nothing but the conditional expectation $\Phi_\DD$ from Theorem~\ref{thm:
using the diagonal}.

Thus it remains to prove Proposition~\ref{prop:extend-to-contraction-using-D}. The established strategy is to express $\Phi$ on finite linear combinations of the spanning family $(v_pv_q^*)_{p,q \in S}$ as a cut-down by a suitable projection that will depend on the given linear combination. Fix therefore finite combinations $t_F\in C^*(S)$ and $t_{F, \DD}\in \DD$ as in \eqref{eq:tF-tFD}.  In view of our aim we assume, without loss of generality, that $t_{F,\DD} \neq 0$ holds (otherwise $0$ is a suitable projection). Most of the preparation needed to construct $\Phi$ was done in section~\ref{subsection-LR}. For case (1), Lemma \ref{lem: projection for trivial S*--D} will suffice.

Condition (D2) is relevant when there are non-trivial elements in $S^*$. These units will appear in the sum from \eqref{def:TF3Q1--D} due to right cancellation in $S$: indeed, for $p,q \in A, p \neq q$ satisfying $p_AS = q_AS$ it follows that $pp_A = qp_Ax$ for some $x\in S^*$. By right cancellation,  necessarily $x \neq 1_S$. Thus there are $n \in \N$, $x_1,\dots,x_n \in S^*{\setminus}\{1_S\}$ and $\lambda_1, \dots, \lambda_n \in \C$ such that
\begin{equation}
\begin{array}{lcl}
e_{\sigma_AS}t_{F,3}e_{\sigma_AS}
&=& \sum\limits_{i=1}^n
\lambda_{i}v_{\sigma_A}v_{x_i}v_{\sigma_A}^*.\label{def:TF3Q1mod--D}
\end{array}
\end{equation}

\begin{lemma}\label{lem: projection for non-trivial units--D}
Let $S$ be a cancellative right LCM semigroup such that  $|S^*|>1$  and $S$
satisfies condition (D2). Let $A, F_1$ be as in Lemma~\ref{lem: second
projection--D}. Then there exists $p_F\in \sigma_AS$ such that $e_F:=e_{p_FS}$
satisfies
\begin{enumerate}
\item[(i)] $e_FQ_{F_1{\setminus}A,\emptyset}^e \neq 0$.
\item[(ii)] $\Vert e_{F}Q_{F_1{\setminus}A,\emptyset}^e t_{F}
e_{F}Q_{F_1{\setminus}A,\emptyset}^e\Vert = \|t_{F,\DD}\|.$
\item[(iii)]If $t_{F,\DD}$ is positive, then
$e_{F}Q_{F_1{\setminus}A,\emptyset}^e t_{F}
e_{F}Q_{F_1{\setminus}A,\emptyset}^e=\|t_{F,\DD}\|e_{F}Q_{F_1{\setminus}A,\emptyset}^e.$
\end{enumerate}
\end{lemma}
\begin{proof}
 Since $\JJ(S)$ is independent, $Q_{F_1,A}^e
\neq 0$ is equivalent to $\sigma_AS \cap \biggl(S{\setminus}\bigcup\limits_{q
\in F_1\setminus A}{qS}\biggr) \neq \emptyset$. Applying  (D2) to $\sigma_A$ in place of $s_0, s_1$, the unit $x_1\in S^*{\setminus}\{1_S\}$ and the finite set $F_1\setminus A$ gives
an element $s_2 \in \sigma_AS$ such that $x_1\sigma_A^{-1}s_2S\cap
\sigma_A^{-1}s_2S=\emptyset$ and $s_2S$ has non-empty intersection with
$S{\setminus}\bigcup\limits_{q \in F_1{\setminus}A} qS$. Next, we apply (D2)
to $\sigma_A$ as $s_0$, $s_2$ in place of $s_1$, the unit $x_2$ and $F_1\setminus A$ resulting
in an element $s_3 \in s_2S$ such that $x_2\sigma_A^{-1}s_3S \cap
\sigma_A^{-1}s_3S = \emptyset$ and $s_3S$ has non-empty intersection with
$S{\setminus}\bigcup\limits_{q \in F_1{\setminus}A} qS$. Note that we have
\[x_1\sigma_A^{-1}s_3S \cap \sigma_A^{-1}s_3S \stackrel{s_3 \in s_2S}{\subset}
x_1\sigma_A^{-1}s_2S \cap \sigma_A^{-1}s_2S = \emptyset.\]
Thus, proceeding inductively, we get $s_n \in \sigma_AS$ such that $s_nS$ has
non-empty intersection with $S{\setminus}\bigcup\limits_{q \in
F_1{\setminus}A} qS$ and $x_i\sigma_A^{-1}s_nS\cap
\sigma_A^{-1}s_nS=\emptyset$ for all $i = 1,\dots,n$. This translates to
$Q_{F_1,A}^e \geq e_{s_nS}Q_{F_1{\setminus}A,\emptyset}^e \neq 0$ and
$e_{s_nS}t_{F,3}e_{s_nS} = 0$. Let $p_F := s_n$. Since
\[
e_{F}Q_{F_1{\setminus}A,\emptyset}^e t_{F}
e_{F}Q_{F_1{\setminus}A,\emptyset}^e=e_{F}Q_{F_1{\setminus}A,\emptyset}^e
t_{F, \DD} e_{F}Q_{F_1{\setminus}A,\emptyset}^e,
\]
an application of Lemma~\ref{lem: second projection--D} shows that $p_F$
satisfies (i)$-$(iii).
\end{proof}

\begin{proof}[Proof of Proposition \ref{prop:extend-to-contraction-using-D}]
For any finite linear combination $T_F \subset \text{span}\{~V_pV_q^*~|~p,q
\in S~\}$, consider the corresponding element $t_F \in C^*(S)$. For case
(1), we use Lemma~\ref{lem: projection for trivial S*--D} to obtain a
non-zero projection  $Q_{F_1,A}^e \in \DD$ that satisfies
\[\Vert Q_{F_1,A}^e t_F Q_{F_1,A}^e\Vert = \|\Phi(t_F)\|.\]
Since $\pi_{\DD}$ is injective and $(V_p)_{p\in S},(E_{pS})_{p\in S}$ are
families of
isometries and projections, respectively, satisfying (L1)--(L4), we get
\[\Vert Q_{F_1,A}^E T_F Q_{F_1,A}^E \Vert= \|\Phi(T_F)\|.\]
As $Q_{F_1,A}^E \neq 0$ is a projection, we get
\[\|\Phi(T_F)\| = \|Q_{F_1,A}^E T_F Q_{F_1,A}^E\| \leq \|T_F\|,\]
so $\Phi$ is contractive on a dense subset of $\pi(C^*(S))$. By standard
arguments, it extends to a contraction from $\pi(C^*(S))$ to $\pi(\DD)$.

For case (2), run the same argument with $e_{F}Q_{F_1,A}^e$ given by
Lemma~\ref{lem: projection for non-trivial units--D} as the suitable
replacement for $Q_{F_1,A}^e$.
\end{proof}

\begin{proof}[Proof of Theorem \ref{thm: using the diagonal}]
Since $S$ is a right LCM semigroup, Proposition~\ref{prop: part (i) independent} implies
that condition~\eqref{eq: key to faithfulness independent version} is
equivalent to injectivity of $\pi$ on $\DD$.
Obviously, $\pi|_{\DD}$ is injective whenever $\pi$ is injective, showing the
forward implication in the theorem. To prove the reverse implication, we apply
Proposition~\ref{prop:extend-to-contraction-using-D} to obtain the following
commutative diagram.
\begin{equation}\label{com:diagram DD}
\begin{tikzpicture}[>=stealth, xscale=2, yscale=1.5]
\node (c) at (0,-2) {$\DD$};
\node (b) at (2,0) {$\pi(C^*(S))$};
\node (a) at (0,0) {$C^*(S)$}
edge[->](b)
edge[->] (c);
\node (d) at (2,-2) {$\pi|_{\DD}(\DD)$}
edge[<-] (b)
edge[<-] (c);
\node at (1,0.2) {$\pi$};
\node at (-0.2,-1) {$\Phi_\DD$};
\node at (2.2,-1) {$\Phi$};
\node at (1,-2.2) {$\pi|_{\DD}$};
\end{tikzpicture}
\end{equation}
Now, if $a \in C^*(S)_+, a \neq 0$, then $\Phi{\circ}\pi(a) =
\pi|_\DD{\circ}\Phi_\DD(a) \neq 0$ as $\Phi_\DD$ is faithful and $\pi|_\DD$ is
injective. Thus, we have $\pi(a) \neq 0$. Since injectivity of *-homomorphisms
can be detected on positive elements, $\pi$ is seen to be injective.
\end{proof}

\section{\texorpdfstring{Purely infinite simple $C^*(S)$ arising from right LCM semigroups}{Purely infinite simple C*(S) arising from right LCM semigroups}}\label{section:pisimple}
\noindent Suppose that $S$ is a right LCM semigroup. Consider the following refinement of condition (D2):
\begin{enumerate}
\item[(D3)] If $s\in S$ and  $F$ is a finite subset of $S$ with $sS \cap \big(S\setminus
\bigcup\limits_{q \in F} {qS}\big) \neq \emptyset$, then there is $s' \in sS$
such that $s'S \cap qS = \emptyset$ for all $q \in F$.
\end{enumerate}

Whenever $F \neq \emptyset$, we will denote its elements by $q_1,\dots,q_n$. In section~\ref{sec: examples} we will see examples of semigroups satisfying conditions (D3) and (D2). To clarify the relationship between (D2) and (D3), we
make the following observation:

\begin{lemma}\label{lem:str eff+D1+D3 gives D2}
Let $S$ be a right LCM semigroup with $S^*\neq \emptyset$. If the action
$S^*\curvearrowright\JJ(S)$ is strongly effective and $S$ satisfies (D1) and
(D3), then (D2) holds.
\end{lemma}
\begin{proof}
We saw in Remark~\ref{rmk:consequence-strongly-effective} that the condition (D2) where $F=\emptyset$ is satisfied
when $S^*\curvearrowright\JJ(S)$ is strongly effective and $S$ satisfies (D1). Thus it remains to prove (D2) in case $F \neq \emptyset$. Let therefore  $s_0,q_1,\dots,q_n \in S$ and $s_1 \in s_0S$ with $s_1S \cap
\biggl(S{\setminus}\bigcup\limits_{i=1}^n q_iS\biggr) \neq \emptyset$.
Applying (D3) yields an element $s \in s_1S \subset s_0S$ such that $sS
\cap q_iS = \emptyset$ for $i=1,\dots, n$. Note that every $s_2 \in sS$
inherits this property, and therefore the first equation in (D2) is satisfied for such
elements. Let $x \in S^*{\setminus}\{1_S\}$ and write $s=s_0r$ for some $r\in S$. By strong
effectiveness and (D1) applied to $x$ and  $rS$ we get $s' \in rS$ with
the property that $s'S\cap xs'S=\emptyset$. Now $s_2=s_0s' \in sS \subset
s_1S$ satisfies $s_0^{-1}s_2S\cap xs_0^{-1}s_2S=s'S\cap xs'S=\emptyset$,
proving (D2).
\end{proof}

\noindent We note that (D1), (D3) and strong effectiveness are properties of a semigroup that can be more readily verified than (D2). The latter condition is quite close to the operator algebraic application it is designed for. Therefore, in Theorem~\ref{thm: simple and p.i.} we provide an independent proof for the last two sets of assumptions, even though (3) may be deduced from the proof in the case (2). First we need a lemma.

\begin{lemma}\label{lem: projection for simplicity (D1)+str.eff+(D3).--D}
Let $S$ be a cancellative right LCM semigroup such that $|S^*|>1$, the action
$S^*\curvearrowright\JJ(S)$ is strongly effective, and (D1), (D3) are
satisfied. Suppose $t_F\in C^*(S)$ and $t_{F,\DD}\in \DD$ are linear combinations as
in \eqref{eq:tF-tFD}, and assume that $t_F$ is a positive element in $C^*(S)$. Then $t_{F,\DD}$ is positive and there is $p_F \in S$ such that $e_F:=e_{p_FS}$ satisfies
\[e_{F} t_F e_{F} = \|t_{F,\DD}\|e_F.\]
\end{lemma}
\begin{proof}
As $t_{F,\DD}$ is the image of $t_F$ under the natural conditional expectation $C^*(S) \to \DD$, it is also positive. According to Lemma~\ref{lem: second projection--D}, there are finite subsets
$A,F_1$ of $S$ with $A \subset F \subset F_1$ and  $Q_{F_1,A}^e \neq 0$ such
that
\[
Q_{F_1,A}^e t_F Q_{F_1,A}^e = Q_{F_1,A}^e(t_{F,\DD} + t_{F,3})Q_{F_1,A}^e
\text{ and }
Q_{F_1,A}^e t_{F,\DD} Q_{F_1,A}^e = \|t_{F,\DD}\|Q_{F_1,A}^e.
\]
 Since $|S^*|>1$, the collection $\JJ(S)$ is independent by Corollary~\ref{cor: independence for monoids}. According to Proposition~\ref{prop: part (i) independent}, we have $Q_{F_1,A}^e \neq 0$ if and only if
 $\sigma_AS \cap \biggl(S{\setminus}\bigcup\limits_{q \in F_1{\setminus}A}
 qS\biggr) \neq \emptyset$.
By (D3), there is $p_0 \in \sigma_AS$ such that $p_0S \cap qS = \emptyset$ for
all $q \in F_1{\setminus}A$. Hence, we have
\[
e_{p_0S} t_F e_{p_0S} = e_{p_0S}(t_{F,\DD} + t_{F,3})e_{p_0S}
\text{ and }
e_{p_0S} t_{F,\DD} e_{p_0S} = \|t_{F,\DD}\|e_{p_0S}.
\]
Let $x_1,\dots,x_n \in S^*{\setminus}\{1_S\}$ be the invertible elements that
appear in \eqref{def:TF3Q1mod--D} and let $p_0'$ denote the element satisfying
$p_0 = \sigma_Ap_0'$.
%Note that we only have to run our argument for $x_i$ satisfying $x_ip_0'S
%\cap p_0'S \neq \emptyset$.
Applying strong effectiveness to $p_0'$ and $x_1$ yields an element $p_1' \in
p_0'S$ satisfying $x_1p_1'S \neq p_1'S$. By (D1), this amounts to $x_1p_1'S
\cap p_1'S = \emptyset$. Proceeding inductively, where $p_i \in p_{i-1}S$ is
given as $p_i = \sigma_Ap_i'$ with $p_i' \in p_{i-1}'S$ satisfying $x_ip_i'S
\cap p_i'S = \emptyset$, we obtain $p_n \in \sigma_AS$ such that
\[\begin{array}{lclr}
e_{p_nS}v_{\sigma_A}v_{x_i}v_{\sigma_A}^*e_{p_nS} &=&
v_{p_n}v_{p_n'}^*v_{x_i}v_{p_n'}v_{p_n}^*\\
&=& v_{p_n}v_{p_n'}^*e_{p_i'S}v_{x_i}e_{p_i'S}v_{p_n'}v_{p_n}^* &(p_n'S
\subset p_i'S)\\
&=& v_{p_n}v_{p_n'}^*e_{p_i'S \cap x_ip_i'S}v_{x_i}v_{p_n'}v_{p_n}^*\\
&=&0
\end{array}\]
for all $i = 1,\dots,n$. Thus, $p_F:= p_n$ satisfies the claim of the lemma.
\end{proof}

\begin{thm}\label{thm: simple and p.i.}
Let $S$ be a cancellative right LCM semigroup such that
$\Phi_{\DD}:C^*(S)\to \DD$ is a faithful conditional expectation.
Assume that (D3) and one of the following conditions hold:
\begin{enumerate}[(1)]
\item $|S^*| \leq 1$.
\item $|S^*|>1$ and $S$ satisfies condition (D2).
\item $|S^*|>1$, $S$ satisfies condition (D1), and the action
$S^*\curvearrowright\JJ(S)$ is strongly effective.
\end{enumerate}
Then $C^*(S)$ is purely infinite and simple.
\end{thm}
\begin{proof}
Recall from Lemma~\ref{lem: spanning family for algebra} that the linear span of the elements $v_pv_q^*$ is dense in $C^*(S)$. Every element from this
linear span has the form $t_F = \sum\limits_{p,q \in F}\lambda_{p,q} v_pv_q^*$
for some finite $F \subset S$ and suitable $\lambda_{p,q} \in \C$. Moreover,
$\Phi_\DD(t_F) = t_{F,\DD} = \sum\limits_{p \in F}\lambda_{p,p} e_{pS}$.

Let $a \in C^*(S)$ be positive and non-zero, and let $\varepsilon > 0$. Choose a
positive linear combination $t_F$  that approximates $a$ up to within  $\varepsilon$. If
$\varepsilon$ is sufficiently small, we have $t_F \neq 0$ which we will assume
from now on. For the four different cases in the hypothesis of the theorem, we
will use different methods to obtain a suitable small projection
$e'_F:=e_{q_FS}$ that annihilates the off-diagonal terms of $t_F$ while
picking up the norm of the diagonal part: that is,
\[e'_{F} t_F e'_{F} = \|t_{F,\DD}\|e'_{F} = \|\Phi_\DD(t_{F})\|e'_{F}.\]
For case (1), we use Lemma~\ref{lem: projection for trivial S*--D},
and for case (2) Lemma~\ref{lem: projection for non-trivial units--D} to get a
finite subset $F_2 = F_1\setminus A$ of $S$ and an element $p_F\in S$ such
that $e_F=e_{p_FS}$ satisfies
\begin{enumerate}
\item[(i)] $e_{F}Q_{F_2,\emptyset}^e \neq 0$, and
\item[(ii)] $e_{F}Q_{F_2,\emptyset}^e t_{F} e_{F}Q_{F_2,\emptyset}^e =
\|t_{F,\DD}\|e_{F}Q_{F_2,\emptyset}^e.$
\end{enumerate}
 Since $S$ is right LCM, (i) translates to $p_FS \cap \biggl(S
 \setminus \bigcup\limits_{q \in F_2} qS\biggr) \neq \emptyset$ according to Lemma~\ref{lem:lambda-iso-diagonal-general}. So we can
 apply (D3) to get an element $q_F \in p_FS$ such that $q_FS \cap qS =
 \emptyset$ for all $q \in F_2$. By (L4) this gives $e_{q_FS} \leq
 e_FQ_{F_2,\emptyset}^e$. Now $e'_F:=e_{q_FS}$ satisfies $ e'_Ft_{F} e'_{F} =
 \|t_{F,\DD}\| e'_F$ by (ii). For case (3), the existence of such a projection $e'_F$ follows directly from
 Lemma~\ref{lem: projection for simplicity (D1)+str.eff+(D3).--D}.

 We have $\|\Phi_\DD(t_F)\| > 0$ since $\Phi_\DD$ is faithful. Thus,
 $e'_{F}t_Fe'_{F}=\|\Phi_\DD(t_F)\|e'_{F}$ is invertible in the corner
 $e'_{F}C^*(S)e'_{F}$. If $\|a-t_F\|$ is sufficiently small, this implies that $e'_{F}ae'_{F}$ is positive and invertible in $e'_{F}C^*(S)e'_{F}$ as well,
 because $\|\Phi_\DD(t_F)\|~\stackrel{\varepsilon \searrow
 0}{\longrightarrow}~\|\Phi_\DD(a)\|~>~0$. Hence, if we denote its positive inverse by
 $b$,	we get
\[\left(b^{\frac{1}{2}}v_{q_F}\right)^{*}e'_{F}ae'_{F}\left(b^{\frac{1}{2}}v_{q_F}\right)
 = v_{q_F}^*e'_{F}v_{q_F} = 1.\]
This implies that $C^*(S)$ is purely infinite and simple.
\end{proof}

\section{\texorpdfstring{Injectivity of the left regular representation of $C^*(S)$}{Injectivity of the left regular representation of C*(S)}}\label{section:injregular}
\noindent A major question of interest in \cite{Li1} and \cite{Li2} is to determine
conditions under which the left regular
representation $\lambda$ is an isomorphism $C^*(S) \cong C^*_r(S)$. In the
context of right LCM semigroups, we have identified some classes of semigroups
for which this isomorphism holds.

\begin{prop}\label{cor: isom full and reduced algebra - simple case} Assume
the hypotheses of Theorem~\ref{thm: simple and p.i.}. Then
the left regular representation $\lambda:C^*(S)\to C_r^*(S)$ is an isomorphism.
\end{prop}
\begin{proof} The conclusion follows because in this case $C^*(S)$ is simple.
\end{proof}

\begin{thm}\label{cor:easy-amenability}
Assume that $S$ is a cancellative right LCM semigroup such that the
conditional expectation $\Phi_\DD$ is faithful. Then the left regular
representation $\lambda$ is an isomorphism from $C^*(S)$ onto  $C^*_r(S)$.
\end{thm}
\begin{proof} By Lemma~\ref{lem:lambda-iso-diagonal-general}, $\lambda$
restricts to an isomorphism $\DD \cong \DD_r$.
Then for $a \in C^*(S)_+, a \neq 0$, we have $\lambda|_\DD \circ
\Phi_\DD(a) \neq 0$. From $\lambda|_\DD \circ \Phi_\DD = \Phi_{\DD,r} \circ \lambda$
and faithfulness of $\Phi_{\DD,r}$ established in Proposition~\ref{prop: faithful cond exp for the red algebra} it follows that $\lambda(a) \neq 0$.
Hence, $\lambda$ is an isomorphism.
\end{proof}

\begin{example} By Theorem~\ref{cor:easy-amenability} and
Proposition~\ref{prop:left-inverse} it follows that $\lambda$ is an isomorphism in the case of
$S=G\rtimes_\theta P$ that is right LCM, right reversible, and satisfies that $S^{-1}S$ is
an amenable group.
\end{example}

\begin{remark}
There is an alternative approach to injectivity of $\lambda$ for certain
subsemigroups of amenable groups, see \cite{Li2}. We refer to  \cite[Section
4]{Li2} for the definition of the Toeplitz condition. Namely, if $S$ is a left
cancellative semigroup satisfying the conditions
 \begin{enumerate}[i)]
 \item $\JJ(S)$ is independent,
 \item $S$ embeds into an amenable group $H$ such that $S$ generates $H$ and
 $S \subset H$ satisfies the Toeplitz condition,
 \end{enumerate}
 then  $\lambda: C^*(S)\longrightarrow C^*_r(S)$ is an isomorphism, cf. the
 equivalence of (iii) and (v) in \cite[Theorem 6.1]{Li2} applied to $A = \C$
 (where (v) is valid because $H$ is amenable). The proof of
 \cite[Theorem 6.1]{Li2} depends on a relatively involved machinery of various
 crossed product constructions. The conclusion in
 Theorem~\ref{cor:easy-amenability} for right LCM semigroups  is obtained through an analysis solely of
 the semigroup $C^*$-algebra $C^*(S)$.
\end{remark}

\section{\texorpdfstring{A uniqueness result using $\CC_I$}{A uniqueness result using the inner core}}\label{section:usingcore}
\noindent In this section we consider left-cancellative semigroups with identity that
satisfy condition (C1). By following an idea from \cite{Mi} we show that it is
possible to construct a conditional expectation from $C^*(S)$ onto $\CC_I$
which may be used to reduce the question of injectivity of representations.

The proof of the next result is similar to \cite[Proposition 3.5]{CLSV}. For a discrete group $\Gamma$, we denote $i_\Gamma$ the canonical homomorphism sending $\gamma$ in $\Gamma$ to the generating unitary $i_\Gamma(\gamma)$ in $C^*(\Gamma)$, and we let $\delta_\Gamma$ be the homomorphism $C^*(\Gamma)\to C^*(\Gamma)\otimes C^*(\Gamma)$ induced by the map $\gamma\to i_\Gamma(\gamma)\otimes i_\Gamma(\gamma)$.

\begin{prop}\label{prop:coaction} Let $S$ be a right LCM semigroup with identity and assume
that there exists a homomorphism $\sigma:S\to T$  onto a
subsemigroup $T$ of a group $\Gamma$ such that $T$
generates $\Gamma$. Then there is a coaction $\delta:C^*(S)\to C^*(S)\otimes
C^*(\Gamma)$  such that
\[\delta(v_pv_q^*)=v_pv_q^*\otimes i_\Gamma(\sigma(p)\sigma(q)^{-1})\]
for all $p,q\in S$.
\end{prop}
\begin{proof} Since $C^*(S)=\clsp\{v_pv_q^*:p,q\in S\}$, we define a map from
$S$ to $C^*(S)\otimes C^*(\Gamma)$ by $s\mapsto v_s\otimes
i_\Gamma(\sigma(s))$ for $s\in S$. A routine calculation shows that
$\{v_s\otimes i_\Gamma(\sigma(s))\}_{s\in S}$ and $\{v_pv_p^*\otimes
i_\Gamma(1_\Gamma)\}$ satisfy the relations that characterise $C^*(S)$, hence
by the universal property we obtain the required homomorphism $\delta$ such
that $\delta(v_s)=v_s\otimes i_\Gamma(\sigma(s))$ for all $s\in S$. Since
$C^*(S)$ is unital and $\delta(1)=1\otimes i_\Gamma(1_\Gamma)$, the map
$\delta$ is nondegenerate. If we let $\varepsilon:C^*(\Gamma)\to \C$ be the
homomorphism integrated from $\gamma\mapsto 1$, then the equality
$(\operatorname{id}_{C^*(S)}\otimes \varepsilon)\circ\delta=
\operatorname{id}_{C^*(S)}$ shows that $\delta$ is an injective map. The
coaction identity
\[
(\delta\otimes \operatorname{id}_{C^*(\Gamma)})\circ
\delta=(\operatorname{id}_{C^*(S)}\otimes \delta_\Gamma)\circ \delta
\]
is immediately checked on the generators $v_s$ of $C^*(S)$.
\end{proof}

\noindent By standard theory of coactions, see for example \cite{Qui}, the spectral
subspaces for $\delta$ are defined as $C^*(S)^\delta_{\sigma(s)}=\{a\in
C^*(S)\mid \delta(a)=a\otimes i_\Gamma (\sigma(s))\}$. The space
\[
C^*(S)^\delta=\{a\in C^*(S)\mid \delta(a)=a\otimes i_\Gamma (1_\Gamma)\}
\]
 is a $C^*$-subalgebra of $C^*(S)$, called the fixed-point algebra. There is
 always a conditional expectation $\Phi^\delta:C^*(S)\to C^*(S)^\delta$ such
 that $\Phi^\delta(a)=0$ if $a\in C^*(S)^\delta_{\sigma(s)}$ with
 $\sigma(s)\neq 1_\Gamma$. Moreover, it is known that $\Phi^\delta$ is
 faithful on positive elements precisely when the coaction $\delta$ is
 \emph{normal}: this is for instance the case when $\Gamma$ is amenable. Since
 $C^*(S)=\clsp\{v_pv_q^*:p,q\in S\}$, we have
 $C^*(S)^\delta=\clsp\{v_pv_q^*:p,q\in S, \sigma(p)=\sigma(q)\}$ and
 \begin{equation}\label{eq:cond-exp}
 \Phi^\delta(v_pv_q^*)=\begin{cases}v_pv_q^*,&\text{ if }\sigma(p)=\sigma(q)\\
 0,&\text{ otherwise}.\end{cases}
 \end{equation}
We will prove in Corollary~\ref{cor:phi-CI} that the above fixed-point algebra
$C^*(S)^\delta$ coincides with $\CC_I$ when the semigroup $S$ satisfies
condition (C1).

\begin{remark}\label{rmk:CO-CI} Recall that  $\CC_I$  was defined as $ C^*(\{v_x,e_{pS}~|~p \in
S,x \in S^*\})$, where we assume $S^*$ is non-trivial, and that  basic
properties of this subalgebra were established in Lemma~\ref{lem: properties
of the subalgebras}. We claim that if $S^*$ is non-trivial and $S$ satisfies
(C1), then $\CC_O=\CC_I$. To see this, note that Lemma~\ref{lem: properties of
the subalgebras} implies that it suffices to show $v_pv_xv_p^*\in\CC_I$ for
all $p \in S$ and $x \in S^*$. By (C1) there exists $y\in
S^*$ such that $px=yp$. Hence,
\[v_pv_xv_p^*=v_ye_{pS}=e_{ypS}v_y\in\CC_I.\]
\end{remark}

\noindent Given a right LCM semigroup $S$ with $S^*\neq \emptyset$ and satisfying (C1),
suppose that the monoid $\mathcal{S}$  constructed in
Proposition~\ref{prop:congruence} embeds into a group $\Gamma$ such that
$\mathcal{S}$ generates $\Gamma$. Proposition~\ref{prop:coaction} applied to
the canonical homomorphism $\sigma:S\to \mathcal{S}$, $\sigma(p)=[p]$ for
$p\in S$ gives a coaction of $\Gamma$ with associated conditional expectation
as described in \eqref{eq:cond-exp}. Note that, in this situation,
$\sigma(p)=\sigma(q)$ means precisely that $p=xq$ for some $x\in S^*$. Hence
$v_pv_q^*=v_{xq}v_q^*=v_{x}e_{qS}$, which is in $\CC_I$. Thus
$C^*(S)^\delta\subseteq \CC_I$, and since the reverse inclusion is immediate,
the two subalgebras of $C^*(S)$ are equal. In case  $S$ is not right
cancellative, it may happen that $p=xq=x'q$ for different $x,x'$ in
$S^*$. However,
$v_{x}e_{qS}=v_{xq}v_q^*=v_{x'q}v_q^*=v_{x'}e_{qS}$. We summarise these
considerations in the following result.

\begin{cor}\label{cor:phi-CI} Let $S$ be a right LCM semigroup such that
$S^*\neq \emptyset$ and (C1) holds. Assume that $\mathcal{S}$ embeds into a
group $\Gamma$ which is generated by the image of $\mathcal{S}$. Then there is
a well-defined conditional expectation $\Phi_{\CC_I}:C^*(S)\to \CC_I$ such that
\begin{equation}\label{def:cond-exp-CI}
\Phi_{\CC_I}(v_pv_q^*)=\begin{cases}v_xe_{qS},&\text{ if }p=xq\text{ for some
}x\in S^*\\
0,&\text{ if }p\not\sim q.\end{cases}
\end{equation}
If $\Gamma$ is amenable, then $\Phi_{\CC_I}$ is faithful on positive elements.
\end{cor}

\noindent Our main result about injectivity of representations in terms of their restriction to $\CC_I$ is the following theorem.

\begin{thm}\label{thm:uniqueness theorem based on C_I}
Let $S$ be a  cancellative right LCM semigroup with identity $1_S$ such that
$S$ satisfies (C1) and the semigroup $\mathcal{S}$ constructed in Proposition~\ref{prop:congruence} embeds into a group
$\Gamma$ in such a way that $\Gamma$ is generated by $\mathcal{S}$. Assume that the conditional expectation $\Phi_{\CC_I}:C^*(S)
\longrightarrow \CC_I$ constructed in Corollary~\ref{cor:phi-CI} is faithful, and that  there is a  faithful conditional expectation $\Phi_0$ from $\CC_I$
onto $\DD$ such that $\Phi_0(e_{qS}v_x)=\delta_{x, 1_S}e_{qS}$ for all $q\in
S$ and $x\in S^*$. Then a *-homomorphism $\pi:C^*(S) \longrightarrow B$ is
injective if and only if $\pi|_{\CC_I}$ is injective.
\end{thm}
\begin{proof} One direction of the theorem is clear, so assume that $\pi\vert_{\CC_I}$ is injective. We must prove that $\pi$ is injective.

 Let  $\Phi := \Phi_0 \circ \Phi_{\CC_I}$ be the faithful conditional
 expectation from $C^*(S)$ to $\DD$ obtained by composing the two given
 expectations. The idea of the proof is to construct a contraction
 $\Phi^\pi:\pi(C^*(S))\to \DD$ such that $\Phi^\pi\circ \pi=\Phi$. Then the
 injectivity of $\pi$ will follow from a standard argument: let $a\in
 C^*(S)_+$ with $a\neq 0$. From $\Phi^{\pi}(\pi(a))=\Phi({a})$ and
 the fact that $\Phi$ is faithful on positive elements it follows that $\pi(a)\neq 0$.

  Let $F \subset S$ be finite and $t_F \in C^*(S)$  a linear combination of
  $v_pv_q^*, p,q \in F$ with scalars $\lambda_{p,q}$ in $\mathbb{C}$ such that
  $t_F$ is positive and non-zero.  Then $\Phi(t_F)\neq 0$. We have
  \begin{align*}
 \Phi(t_F)
 &=\sum_{p\sim q}\lambda_{p,q} (\Phi_0\circ \Phi_{\CC_I})(v_pv_q^*)\\
 &= \sum_{p\sim q, p=xq}\lambda_{p,q}\Phi_0(v_xe_{qS})\\
 &=\sum_{\{p,q\in F\mid p=q\}} \lambda_{p,q}e_{qS},
  \end{align*}
  which is $t_{F,\DD}$, and is non-zero.

  By Lemma~\ref{lem: second projection--D}, there are finite subsets $A
  \subset F \subset F_1 \subset S$ such that $Q_{F_1,A}^et_FQ_{F_1,A}^e \in
  \CC_O$ and $Q_{F_1,A}^et_{F,\DD} = \|t_{F,\DD}\| Q_{F_1,A}^e \neq 0$. Remark~\ref{rmk:CO-CI}
  shows that $\CC_O = \CC_I$, so  $Q_{F_1,A}^et_FQ_{F_1,A}^e \in \CC_I$.
  Therefore, we get
\[\begin{array}{lcl}
\Phi(Q_{F_1,A}^et_FQ_{F_1,A}^e) &=& Q_{F_1,A}^e \Phi(t_F)Q_{F_1,A}^e\\
&=& Q_{F_1,A}^et_{F,\DD}Q_{F_1,A}^e \\
&=& \|t_{F,\DD}\| Q_{F_1,A}^e \neq 0.
\end{array}\]
Since $\pi|_{C_I}$ is injective, it follows that
\[\begin{array}{lcl}
\|\pi(t_F)\| &\geq& \|\pi(Q_{F_1,A}^e)\pi(t_F)\pi(Q_{F_1,A}^e)\| \\
&=& \|\pi(Q_{F_1,A}^et_FQ_{F_1,A}^e)\| \\
&=& \|Q_{F_1,A}^et_FQ_{F_1,A}^e\| \\
&\geq& \| \Phi(Q_{F_1,A}^et_FQ_{F_1,A}^e)\| \\
&=& \|t_{F,\DD}\|.
\end{array}\]
Thus $\Phi^\pi:\pi(C^*(S))\to \DD$ given by $\Phi^\pi(\pi(t_F)):=t_{F, \DD}$
is a well-defined contraction with the desired properties.
\end{proof}

\noindent For the moment the only examples of semigroups $S$ we know for which the hypotheses of Theorem~\ref{thm:uniqueness theorem based on C_I} are satisfied are semidirect products $G\rtimes_\theta P$ covered by Proposition~\ref{prop:left-inverse}. However, we expect that the theorem will apply in situations where $G$ or the enveloping group of $P$ are non-amenable, for example when they are free groups. The challenge is to generalise the arguments of \cite{LaRa} that show existence of a faithful conditional expectation onto $\DD$ from the case $S^*=\{1_S\}$ to the case that there are non-trivial units.

\begin{cor}\label{cor: isom full and reduced algebras - easy argument}
Assume the notation and hypotheses of Theorem~\ref{thm:uniqueness theorem
based on C_I}.
Then the left regular representation $\lambda: C^*(S) \longrightarrow
C^*_r(S)$ is an isomorphism.
\end{cor}
\begin{proof} Let $\Phi=\Phi_0\circ \Phi_{\CC_I}$ be the faithful conditional
expectation from Theorem~\ref{thm:uniqueness theorem based on C_I}. Since $S$
has an identity, $\JJ(S)$ is independent, and  therefore $\lambda$ restricts
to an isomorphism $\DD \cong \DD_r$, \cite{Li1}. Thus the argument of
Theorem~\ref{cor:easy-amenability} can be used.
\end{proof}

\noindent Results with the flavour of a gauge-invariant uniqueness theorem have been proved for many classes of $C^*$-algebras, see \cite{CLSV} and the references therein. In our context, a straightforward version is as presented in the next proposition.

\begin{prop}\label{thm:giut}
Let $S$ be a cancellative right LCM semigroup with identity such that $S$
satisfies (C1) and the conditional expectation
$\Phi_{\CC_I}:C^*(S) \longrightarrow \CC_I$ constructed in
Corollary~\ref{cor:phi-CI} is faithful. Then a $*$-homomorphism $\pi:C^*(S)\to
B$ is injective if and only if $\pi\vert_{\CC_I}$ is injective and $B$ admits
a coaction $\epsilon$ of the enveloping group $\Gamma$ of $\mathcal{S}$
such that $\pi$ is $\delta-\epsilon$-equivariant, i.e. $(\pi\otimes\operatorname{id}_{C^*(\Gamma)})\circ \delta=\epsilon\circ \pi$.
\end{prop}
\begin{proof} If $B$ admits a coaction $\epsilon$ as in the hypothesis, then there is a
conditional expectation $\Phi^\epsilon$ from $B$ onto $\pi(\CC_I)$ such that
$\Phi^\epsilon\circ \pi=(\pi\vert_{\CC_I})\circ \Phi_{\CC_I}$. Now the
standard argument shows that injectivity of $\pi$ on $\CC_I$ can be lifted to
$C^*(S)$.
\end{proof}

\section{Applications}\label{sec: examples}
\subsection{Semidirect products of groups by semigroups}~\\
The new class of right LCM semigroups covered in this work is that of
semidirect products of a group by the action of a semigroup. Throughout this
subsection, let $G$ be a group, $P$ a cancellative right LCM semigroup with identity $1_P$,
and $P\stackrel{\theta}{\curvearrowright}G$ an action by injective group
endomorphisms of $G$. The semidirect product $\gxp$ is denoted by $S$ for
convenience of notation.

\begin{definition}
The action $\theta$ is said to \emph{respect the order} on $P$ if for all $p,q \in
P$ with $pP \cap qP \neq \emptyset$, we have $\theta_p(G) \cap
\theta_q(G) = \theta_r(G)$, where $r\in P$ is any element such that $pP \cap
qP = rP$. This is well-defined since $r_1P= pP
\cap qP=r_2P$ implies that $r_1 = r_2x$ for some $x \in
P^*$, which means $\theta_{r_1}(G)=\theta_{r_2x}(G)=\theta_{r_2}(G)$ because
$\theta_x$ is an automorphism of $G$.
\end{definition}

\begin{prop}\label{prop: GxP right LCM}
If $\theta$ respects the order, then $S$ is a right LCM semigroup.
\end{prop}
\begin{proof} Since both $G$ and $P$ are left cancellative and $\theta$ acts
by injective maps, $S$ is left cancellative.
Suppose $g_1,g_2 \in G$ and $p_1,p_2 \in P$ such that $(g_1,p_1)S \cap
(g_2,p_2)S \neq \emptyset$. Then $p_1P \cap p_2P \neq \emptyset$,
and since $P$ is right LCM, there is $q \in P$ satisfying $p_1P \cap p_2P =
qP$. Denote by $p_1',q_1'\in P$ the elements satisfying $q=p_1p_1'=p_2p_2'$.
We must also have $h_1,h_2\in G$ such that
$g_1\theta_{p_1}(h_1)=g_2\theta_{p_2}(h_2)$. We claim that
\[(g_1,p_1)S \cap (g_2,p_2)S = (g_1\theta_{p_1}(h_1),q)S.\]
Since $(g_1\theta_{p_1}(h_1),q)=(g_1,p_1)(h_1,p_1')=(g_2,p_2)(h_2,p_2')$, the right ideal $(g_1\theta_{p_1}(h_1),q)S$ is contained in $(g_1,p_1)S \cap (g_2,p_2)S$.

For the reverse containment, suppose that $(g,s),(h,t)\in S$ with
$(g_1,p_1)(g,s)=(g_2,p_2)(h,t)$. Then $g_1\theta_{p_1}(g)=g_2\theta_{p_2}(h)$
and $p_1s=p_2t$. We now immediately have $p_1s=p_2t=qq'$ for some $q'\in P$.
The identities $g_1\theta_{p_1}(h_1)=g_2\theta_{p_2}(h_2)$ and
$g_1\theta_{p_1}(g)=g_2\theta_{p_2}(h)$ yield
$\theta_{p_1}(h_1^{-1}g)=\theta_{p_2}(h_2^{-1}h)$. Since $\theta$ respects the
order on $P$, we have $\theta_{p_1}(G)\cap\theta_{p_2}(G)=\theta_q(G)$, and
hence $\theta_{p_1}(h_1^{-1}g)=\theta_q(k)$ for some $k\in G$. Then
\begin{align*}
(g_1,p_1)(g,s)=(g_1\theta_{p_1}(g),p_1s)
=(g_1\theta_{p_1}(h_1)\theta_{p_1}(h_1^{-1}g),p_1s)
&=(g_1\theta_{p_1}(h_1),q)(k,q')\\
&\in (g_1\theta_{p_1}(h_1),q)S.
\end{align*}
So the reverse containment holds, and hence $S$ is right LCM.
\end{proof}

\noindent Since the focus of this paper is on right LCM semigroups we shall assume from
now on that $\theta$ respects the order. The structure of $\JJ(S)$ is determined by the semigroup $P$ and the collection of cosets $\{G/\theta_p(G)\}_{p \in P}$.

\begin{lemma}\label{lem: GxP disjoint ideals}
For any $g,h \in G$ and $p \in P$ we have
\[(g,p)S \cap (h,p)S\vspace*{2mm} = \begin{cases} (g,p)S, &\text{if } g^{-1}h
\in \theta_p(G),\\ \emptyset, &\text{ otherwise.} \end{cases}\]
\end{lemma}
\begin{proof}
If the intersection is non-empty, we have $g\theta_p(g_1) = h\theta_p(h_1)$
for some $g_1,h_1 \in G$. Then $g^{-1}h =\theta_p(g_1h_1^{-1}) \in
\theta_p(G)$, as needed.
\end{proof}

\begin{cor}\label{cor:useful-equivalence}
Let $P$ satisfy (C2). For any $(g,p)\in S$ and $(h,x)\in S^*$, the following are equivalent:
\begin{enumerate}
\item[(i)] $(h,x)(g,p)S \neq (g,p)S$;
\item[(ii)]  $(h\theta_x(g),p)S \cap (g,p)S =\emptyset$;
\item[(iii)] $g^{-1}h\theta_x(g) \notin \theta_p(G)$.
\end{enumerate}
In particular, $S$ satisfies condition (D1).
\end{cor}
\begin{proof}
Take $(g,p) \in S$ and $(h,x) \in S^*$. By Lemma \ref{lem: units in GxP}, we
 have that $x \in P^*$. Condition (C2) gives $y\in P^*$ with $xp = py$. Therefore,
 \[ (h,x)(g,p)S\cap (g,p)S=(h\theta_x(g),p)S\cap (g,p)S.\]
By Lemma~\ref{lem: GxP disjoint ideals}, these intersections are non-empty if and only if
$g^{-1}h\theta_x(g)\in \theta_p(G)$, in which case the ideals $(h,x)(g,p)S$ and $(g,p)S$ coincide. It follows immediately from this
that $S$ satisfies condition (D1).
\end{proof}

\begin{lemma}\label{lem: GxP str eff}
Let $P$ satisfy (C2). Then the action $S^*{\curvearrowright}\JJ(S)$ from
Definition~\ref{def: conditions on S for using D} is strongly effective
if and only if it is effective.
\end{lemma}
\begin{proof}
Strong effectiveness implies effectiveness. Assume therefore that $S^*{\curvearrowright}\JJ(S)$ is effective. Let $(g,p) \in S$ and $(h,x) \in (G\times P^*)\setminus\{(1_G,1_P)\}$, where $S^*=G\times P^*$ by Lemma~\ref{lem: units in GxP}. If $(h,x)(g,p)S \neq (g,p)S$ holds, then $(g,p)$ itself does the job required for strong effectiveness.

\noindent Let now $(h,x)(g,p)S = (g,p)S$. We have to find an element $(g',p') \in S$ satisfying
\begin{equation}\label{eq:strong-effect}
(h,x)(g,p)(g',p')S \neq (g,p)(g',p')S.
\end{equation}

It follows from Corollary~\ref{cor:useful-equivalence} that $g^{-1}h\theta_x(g) = \theta_p(\tilde{g})$ for some
$\tilde{g} \in G$. Using (C2) to find $y$ in $S^*$ such that $xpp'=pp'y$, the left-hand side of \eqref{eq:strong-effect} rewrites as
\[
(h,x)(g,p)(g',p')S = (h\theta_x(g)\theta_{xp}(g'),xpp')S =
(h\theta_x(g)\theta_{py}(g'),pp')S.
\]
Thus to prove \eqref{eq:strong-effect}  we need to ensure that
\[
(g\theta_p(g'))^{-1}h\theta_x(g)\theta_{py}(g') =
\theta_p((g')^{-1})g^{-1}h\theta_x(g)\theta_{py}(g') \notin \theta_{pp'}(G).
\]
Since $g^{-1}h\theta_x(g) = \theta_p(\tilde{g})$ and $\theta_p$ is injective,
this is equivalent to
\[(g')^{-1}\tilde{g}\theta_{y}(g') \notin \theta_{p'}(G).\]
Since $x\neq 1_P$ implies, by right cancellation in $P$, that $y\neq 1_P$,  we see that the existence of $(g',p')$ is guaranteed by effectiveness
of the action applied to $(\tilde{g},y) \in S^*\setminus\{1_S\}$. Thus $S^*{\curvearrowright}\JJ(S)$ is strongly effective.
\end{proof}

\noindent Since an action of a group on a space is effective precisely when the intersection of all stabiliser subgroups is the trivial subgroup, Lemma~\ref{lem: GxP str eff} says that we can rephrase the property of  $S^* {\curvearrowright} \JJ(S)$ being strongly effective in terms of stabilisers. We introduce first some notation.  For each $(g,p)\in S$, let $S_{(g,p)}$ denote the subgroup of
$G$ equal to $g\theta_p(G)g^{-1}$. For the action $S^* {\curvearrowright} \JJ(S)$ from
Definition~\ref{def: conditions on S for using D}, let $S^*_{(g,p)S}$ denote the stabiliser subgroup of $(g,p)S\in\JJ(S)$.

\begin{lemma}\label{lem: stab for (GxP)* on J(GxP)}
Let $P$ satisfy (C2) and consider the action $S^* {\curvearrowright} \JJ(S)$ from
Definition~\ref{def: conditions on S for using D}. Then the stabiliser subgroup of
$(g,p)S\in \JJ(S)$ takes the form
\[S^*_{(g,p)S} = \{(h,x) \in S^*\mid h\theta_x(g) \in g\theta_p(G)\}.\]
If  $P^* = \{1_P\}$, then $S^*_{(g,p)S} = S_{(g,p)}\times\{1_P\}$.

Further, $S^* {\curvearrowright} \JJ(S)$ is strongly effective if and only if
\[\bigcap_{(g,p) \in S} S_{(g,p)}= \{1_G\}.\]
In particular, if $P^* = \{1_P\}$ and $G$ is abelian, then $S^* {\curvearrowright} \JJ(S)$ is strongly effective if and only if
$\bigcap_{p \in P} \theta_p(G) = \{1_G\}$.
\end{lemma}
\begin{proof}
The claimed description of $S^*_{(g,p)S}$ follows from Corollary~\ref{cor:useful-equivalence}, and the characterisation of strongly effective follows from Lemma~\ref{lem: GxP str eff}.
\end{proof}

\noindent We shall be able to say more for semigroups $S$ where $P$ is a countably generated free abelian semigroup with identity. For the purposes of the next results, we therefore assume that $P \cong \N^k$ for some $k \in \N$ or $P \cong \bigoplus_\N \N$. In this case, (C2) is automatic for $P$, hence $S = \gxp$ satisfies (D1) by Corollary~\ref{cor:useful-equivalence}.

As indicated in the comment following Lemma~\ref{lem:str eff+D1+D3 gives D2}, condition (D2) is harder to establish in full generality. The next result describes an obstruction to having (D2) satisfied by $S= \gxp$.

\begin{lemma}\label{lem: GxP and D2}
  Assume $P \cong \N^k$ for some $k \in \N$ or $P \cong \bigoplus_\N \N$. If there are $q_1,\dots,q_m \in P\setminus\{1_P\}$ such that $[G:\theta_{q_i}(G)] < \infty$ and
\[
\bigcap\limits_{\stackrel{(g,p)\in S}{p \in P\setminus (\bigcup_{1 \leq i \leq m}q_iP)}} S_{(g,p)} \supsetneqq \{1_G\},
\]
then $S$ does not satisfy (D2).
\end{lemma}
\begin{proof}
Suppose there are $q_1,\dots,q_m$ as prescribed above and pick an element
\[
h \in \bigcap\limits_{\stackrel{(g,p)\in S}{p \in P\setminus (\bigcup_{1 \leq i \leq m}q_iP)}} S_{(g,p)}
\]
with $h \neq 1_G$.  Denote $[G:\theta_{q_i}(G)] = N_i$ in $\mathbb{N}^\times$, and choose, for each $i=1,\dots, m$,   a complete set of representatives $(h_{i,j})_{1 \leq j \leq N_i}$ for $G/\theta_{q_i}(G)$.

We claim  that (D2) fails for the choice of elements $(g_0,p_0)=(g_1,p_1)=(1_G,1_P)$ in $S$, $(h,1_P)$ in $S^*$, and the finite subset $\{(h_{i,j},q_i)\mid j=1, \dots ,N_i,\, i=1,\dots , m\}$ of $S$. Note that we have $(g_1,p_1) \notin (h_{i,j},q_i)S$ for all $j=1, \dots ,N_i$ and $i=1,\dots , m$. If (D2) were to hold, it would imply the existence of $(g_2,p_2)$ such that both $(g_2,p_2) \notin (h_{i,j},q_i)S$ for all $i,j$ and, by Corollary~\ref{cor:useful-equivalence}, also $h\notin S_{(g_2,p_2)}$. Hence, by
 the choice of $h$, there is at least one $i$ with $p_2 \in q_iP$. For this $i$ there is a unique $j \in \{1,\dots,N_i\}$ with $g_2 \in h_{i,j}\theta_{q_i}(G)$. In other words, we would get $(g_2,p_2) \in (h_{i,j},q_i)S$, which is a falsehood.
\end{proof}

\noindent In the next two examples we describe some situations where $S$ satisfies condition (D2).

\begin{example}\label{ex: GxP sat D2 but not D3 I}
Let $G = \bigoplus_\N \Z$ and $P$ be the  unital subsemigroup of  $\N^\times$ generated by $2$ and $3$. We shall denote  $P=| 2,3\rangle$. Define an action $\theta$ of $P$ by injective endomorphisms of $G$ as follows: for $g = (g_n)_{n \in \N} \in G$, let
\[
\theta_2(g) = 2g,\,(\theta_3(g))_0 = 3g_0 \text{ and }(\theta_3(g))_n = g_n \text{ for all }n \geq 1.
\]
It is immediate that $\theta$ preserves the order, so $S$ is right LCM by Proposition~\ref{prop: GxP right LCM}. Further, $[G:\theta_2(G)] = \infty$ and $[G:\theta_3(G)] = 3$. Note that $\bigcap_{n \in \N} \theta_{2^n}(G) = \{1_G\}$. We claim that $S=\gxp$ satisfies (D2). It will follow from Lemma~\ref{lem: GxP D2} that $S$ does not satisfy (D3).

Suppose that we have $s_0:=(g_0,p_0) \in S,s_1:=(g_1,p_1) \in (g_0,p_0)S$ as well as\linebreak $(h_1,q_1),\dots,(h_m,q_m) \in S$ such that
\[
(g_1,p_1)S \cap \biggl(S\setminus\bigcup_{1 \leq i \leq m}(h_i,q_i)S\biggr) \neq \emptyset.
\]
In particular, this implies $(g_1,p_1) \notin (h_i,q_i)S$ for each $i=1, \dots, m$. In case  $p_1 \in q_iP$ for some $i$, then necessarily $g_1 \notin h_i\theta_{q_i}(G)$,  and therefore Lemma~\ref{lem: GxP disjoint ideals}  implies that $(g_1,p_1)S \cap (h_i,q_i)S = \emptyset$. Without loss of generality we may thus assume that $p_1 \notin q_iP$ for all $i=1, \dots, m$.

Let now $x=(g,1_P)$ with $g \neq 1_G$. An element $s_2:=(g_2, p_2)$ as required in (D2) will have to satisfy $xs_0^{-1}s_2S\cap s_0^{-1}s_2S=\emptyset$. If we denote $r:=p_0^{-1}p_2$, this requirement takes the form $(g,1_P)(h',r)S\cap (h',r)S=\emptyset$ for some $h'\in G$. Now, using $\bigcap_{n \in \N} \theta_{2^n}(G) = \{1_G\}$, we can choose $n \in \N$ large enough so that $p_2 := p_1 2^n$ satisfies $g \notin \theta_{r}(G)$.  By Corollary~\ref{cor:useful-equivalence}, this means that $x(h', r)S\cap (h',r)S=\emptyset$ for any $h'\in G$. Thus we have freedom to choose the first entry in $s_2$, and this choice must be made so that it ensures the second requirement in (D2). The crucial ingredient here is the fact that $[G:\theta_{2^k}(G)] = \infty$ for all $k\geq 1$, which will allow us to choose $g_2 \in g_1\theta_{p_1}(G)$ such that $(g_2,p_2) \notin (h_i,q_i)S$ for all $i$. To achieve this goal requires a careful  argument.

If $p_2\notin q_iP$ for all $i=1, \dots, m$, then any choice of $g_2\in g_1\theta_{p_1}(G)$ will ensure that $(g_2,p_2)\notin (h_i,q_i)S$ for all $i$. Assume next that $q_1,\dots,q_m$ are labelled in such a way that there is $m' \in \{1, \dots , m\}$ with the property that $p_2 \in q_iP$ implies $i\leq m'$.  Note that the elements corresponding to $i = m'+1,\dots,m$ pose no obstruction to the choice of $g_2$ because for these indices $i$ we have $(g_2,p_2)\notin (h_i, q_i)S$ irrespective of the choice of $g_2$.  Possibly changing enumeration once more, we can assume that $1$ is minimal in $\{1,\dots,m'\}$ in the sense that
\[
p_1P \cap q_1P \subset p_1P \cap q_iP \Longrightarrow p_1P \cap q_1P = p_1P \cap q_iP \text{ for all } 1 \leq i \leq m',
\]
and that $2,\dots,m'$ are assigned in such a way that there is $m_1$ with
\[p_1P \cap q_1P = p_1P \cap q_iP \Longrightarrow i \leq m_1 \text{ for } i \in 2,\dots,m'.
\]

 Let $ n_1$ be such that $p_1P \cap q_1P = p_1 2^{n_1}$, and note that $1\leq n_1 \leq n$. Since $[G:\theta_{2^{n_1}}(G)] = \infty$,  there are infinitely many distinct principal right ideals of the form $(g_1\theta_{p_1}(g'),p_1 2^{n_1})S \subset (g_1,p_1)S$ with  $g' \in G$. Since $S$ is right LCM, of these infinitely many ideals, at most $m_1$ of them are not admissible for a choice of $g_2$ (because they are possibly contained in $(h_1,q_1)S,\dots,(h_{m_1},q_{m_1})S$). Thus there is $g_{2,1} \in g_1\theta_{p_1}(G)$ with
\[(g_{2,1},p_1 2^{n_1}) \notin (h_i,q_i)S \text{ for all } i = 1,\dots,m_1.\]
Replacing $g_1$ by $g_{2,1}$, $p_1$ by $p_1 2^{n_1}$, $n$ by $n-n_1$, and $\{1,\dots,m'\}$ by $\{m_1+1,\dots,m'\}$, we can iterate this process. Thus at the second step we obtain an element $(g_{2,2}, p_12^{n_1}2^{n_2})\in (g_{2,1},p_1 2^{n_1})S$, for appropriate $ 1\leq n_2\leq n$ and $g_{2,2}\in G$, which also avoids additional ideals $(h_i, q_i)S$, where $i\in \{1, \dots, m_2\}$ is a subset of $\{m_1+1,\dots,m'\}$ for appropriate $m_2$.  This  process stops after finitely many steps (equal to $m\geq 1$ if $n_1+\cdots +n_m=n$) because $n$ is finite. Hence the final pair $(g_{2,m},p_2)$ has the required properties.
\end{example}

\noindent The second example shows that we can also have (D2) in the absence of endomorphisms with infinite index:

\begin{example}\label{ex: GxP sat D2 but not D3 II}
Let $G = \Z$, $P = \N^\times$ and $\theta$ be given by multiplication, i.e. $\theta_p(g) = pg$. Clearly, we have $[G:\theta_p(G)] = p < \infty$ for all $p \in P$. Also, note that for all $q_1,\dots,q_m \in P\setminus\{1_P\}$, we have
\[\bigcap\limits_{p \in P\setminus (\bigcup\limits_{1 \leq i \leq m}q_iP)} \theta_p(G) = \{1_G\}\]
since $P\setminus (\bigcup_{1 \leq i \leq m}q_iP)$ contains arbitrarily large positive integers. So there is an abundance of subsemigroups $Q \subset P$ for which the restricted action $\theta|_Q$ separates the points in $G$. We claim that $S=\gxp$ satisfies (D2).

Let $(g_0,p_0)\in S$, $(g_1,p_1) \in (g_0, p_0)S$, $(h_1,q_1),\dots,(h_m,q_m) \in S\setminus\{1_S\}$ with $(g_1,p_1) \notin (h_i,q_i)S$ for $i = 1,\dots,m$, and $x=(g,1_P) \in S$ with $g \neq 1_G$. For similar reasons  as in Example~\ref{ex: GxP sat D2 but not D3 I}, we can assume that $p_1 \notin q_iP$ holds for all $i$.

Now choose a prime $p \in P$ that does not divide any of the $q_1,\dots,q_m$. Take $n \geq 1$ such that $g \notin \theta_{p^n}(G) = p^n\Z$. If we let $p_2 := p_1p^n$ and $g_2 := g_1$, then $p_2 \notin q_iP$ and hence $(g_2,p_2) \notin (h_i,q_i)S$ for all $i$. Moreover, $p_0^{-1}p_2 \in p^nP$ as $p_1 \in p_0P$. Therefore  $g \notin \theta_{p_0^{-1}p_2}(G)$. Hence  Corollary~\ref{cor:useful-equivalence} implies that $(g,1_P)(g_0,p_0)^{-1}(g_2,p_2)S \cap (g_0,p_0)^{-1}(g_2,p_2)S = \emptyset$, showing  (D2).
\end{example}

\begin{remark}
One can relax the assumptions and consider semidirect products of suitable
semigroups by semigroups instead, for instance positive cones $G_+$ in a group
$G$ on which we already have an action $\theta$ of a semigroup $P$. A natural
assumption in this setting would be $\theta_p(G_+) \subset G_+$. Natural
examples of this kind arise for $\N^k \subset \Z^k$ where $\theta$ takes
values in $\operatorname{M}_k(\mathbb{N})\cap \operatorname{GL}_k(\Q)$.\vspace*{3mm}
\end{remark}

\subsection{\texorpdfstring{Examples of purely infinite simple semigroup $C^*$-algebras from semidirect products}{Examples of purely infinite simple semigroup C*-algebras from semidirect products}}~\\
As before, we consider $S$ of the form $\gxp$, where $P$ is a countably generated, free abelian
semigroup with identity  and $P\stackrel{\theta}{\curvearrowright}G$ is an action by injective group
endomorphisms of $G$ that respects the order.  In this subsection, we show that
Theorem~\ref{thm: simple and p.i.}~(3) applies to $S$ if
$[G:\theta_p(G)]$ is infinite for every $p \neq 1_P$. We illustrate the theorem with several concrete examples of  semigroups
whose semigroup $C^*$-algebra is purely infinite and simple.

\begin{lemma}\label{lem: GxP D2}
Assume that $P \cong \N^k$ for some $k \geq 1$ or $P \cong \bigoplus_\N \N$. Then $S$ satisfies (D3) if and only if the index $[G:\theta_p(G)]$ is infinite for every $p \neq 1_P$.
\end{lemma}
\begin{proof} To begin with, note that $S^* = G\times\{1_P\}$.
Suppose that there exists $q \neq 1_P$ such that $[G:\theta_q(G)] = n < \infty$ and
let $h_1,\dots,h_n \in G$ be a complete set of representatives for
$G/\theta_q(G)$. We claim that (D3) fails for $(g,p)=(1_G,1_P)=1_S$ and $(h_1,q),\dots,(h_n,q)$. To see this, note first that $(g,p) \notin (h_k,q)S$ for all $1 \leq k \leq n$. If  $(g',p') \in (g,p)S$ is arbitrary, then there is a unique $k$ such that $g' \in h_k\theta_q(G)$, i.e. $g' = h_k\theta_q(\tilde{g})$ for some $\tilde{g} \in G$. Since  $P$ is commutative, we have $
(g',p')S \cap (h_k,q)S \supset (g',p'q)S \cap (h_k\theta_q(\tilde{g}),p'q)S$, which equals $(g',p'q)S$, so $(g',p')S \cap (h_k,q)S$ is non-empty.

Now suppose $[G:\theta_p(G)]$ is infinite for every $p \in
P{\setminus}\{1_P\}$. Let $(g,p) \in S$ and $F
\subset S$ be finite such that
\[(g,p)S \cap \biggl(S\setminus\bigcup\limits_{(h,q) \in F} (h,q)S\biggr) \neq
\emptyset.\]
Without loss of generality, we may assume $(g,p)S \cap (h,q)S \neq \emptyset$
and $p \neq q$ hold for all $(h,q) \in F$. Consider
\[F_P := \{ r \mid pP \cap qP = rP \text{ for some } (h,q) \in F \}.\]
Pick $p_1 \in F_P$ which is minimal in the sense that for any other $r \in
F_P$, $p_1 \in rP$ implies $r=p_1$. Let $(h_1,q_1),\dots,(h_n,q_n) \in F$
denote the elements satisfying $pP \cap q_iP = p_1P$. According to
Proposition~\ref{prop: GxP right LCM}, the fact that $(g,p)S \cap (h_i,q_i)S \neq \emptyset$ for all $i=1,\dots, n$ shows that we have
\[
(g,p)S \cap (h_i,q_i)S = (g\theta_p(g'_i),p_1)S=(g,p)(g_i', p^{-1}p_1)S
\]
for suitable $g'_i \in G$ and each $i=1, \dots, n$. Since $p \neq q_i$, we have $p^{-1}p_1 \neq 1_P$ and hence the index $[G:\theta_{p^{-1}p_1}(G)]$ is infinite.
In particular, there is $g_1 \in g\theta_p(G)$ such that
\[(g_1,p_1) \in (g,p)S \text{ and } (g_1,p_1)S \cap (h_i,q_i)S = \emptyset
\text{ for all } i = 1,\dots,n.\]
Setting
\[F_1:= \{(h,q) \in F\mid (h,q)S \cap (g_1,p_1)S \neq \emptyset \},\]
we observe that $F_1 \subset F\setminus\{(h_1,q_1),\dots,(h_n,q_n)\}$ so $F_1 \subsetneqq F$. If $F_1$ is empty, then we are done so let us assume that $F_1 \neq \emptyset$. Note that the minimal way in which $p_1$ was chosen implies $p_1 \notin pP \cap qP$ for all $(h,q) \in F_1$. This will allow us to conclude
\[
(g_1,p_1)S \cap \biggl(S\setminus\bigcup\limits_{(h,q) \in F_1},
(h,q)S\biggr) \neq \emptyset
\]
by invoking the choice of $(g,p)$ and $F$. Indeed, if the intersection was empty, then there would be $(h,q) \in F_1$ with
$(g_1,p_1)S \subset (h,q)S$, see Proposition~\ref{prop: GxP right LCM}. This would
force $p_1 \in qP$ and therefore $p_1 \in p_1P \cap qP \subset pP \cap qP$,
contradicting $(h,q) \in F_1$. Thus, we can iterate this process and, after finitely
many steps, arrive at an element $(g',p') \in (g,p)S$ with the property
$(g',p')S \cap (h,q)S = \emptyset$ for all $(h,q) \in F$. This completes the proof
of the lemma.
\end{proof}

\begin{thm}\label{thm:gxp p.i. simple}
Suppose $G$ is a group, $P \cong \N^k$ for some $k \geq 1$ or $P \cong
\bigoplus_\N \N$, and $P\stackrel{\theta}{\curvearrowright}G$ is an action by
injective group endomorphisms of $G$ respecting the order. Denote $S=\gxp$. Assume that $\bigcap_{p
\in P} \theta_p(G) = \{1_G\}$, $[G:\theta_p(G)]$ is infinite for every $p
\neq 1_P$ and the conditional expectation $C^*(S) \stackrel{\Phi_\DD}{\longrightarrow} \DD$ is
faithful. Then $C^*(S)$ is purely infinite and simple.
\end{thm}
\begin{proof}
We intend to apply Theorem~\ref{thm: simple and p.i.}~(3). First, note that
(D1) holds by Corollary~\ref{cor:useful-equivalence} since (C2) is trivially satisfied for
$P$. By Lemma~\ref{lem: stab for (GxP)* on J(GxP)}, $\bigcap_{p
\in P} \theta_p(G) = \{1_G\}$ corresponds to strong effectiveness of $S^*{\curvearrowright}\JJ(S)$. The fact that $S$ satisfies
(D3) follows from Lemma~\ref{lem: GxP D2}. Since $\Phi_\DD$ is faithful, Theorem~\ref{thm: simple and p.i.}~(3) implies that $C^*(S)$ is purely infinite and simple.
\end{proof}

\noindent Let us now look at some concrete examples. We start with a shift space:

\begin{example}\label{ex:gxp shift}
Let $P \cong \N^k$ for some $k \geq 1$ or $P \cong \bigoplus_\N \N$ and
suppose $G_0$ is a countable amenable group. To avoid pathologies, let us
assume that $G_0$ has at least two distinct elements. Then $P$ admits a shift
action $\theta$ on $G := \bigoplus\limits_P G_0$ given by
\[(\theta_p((g_q)_{q \in P}))_r = \Chi_{pP}(r)~g_{p^{-1}r} \text{ for all }
p,r \in P.\]
It is apparent that $\theta$ is an action by injective group endomorphism that respects the order and $\bigcap_{p \in P} \theta_p(G) = \{1_G\}$ holds. We note that $S$ is a right reversible semigroup whose enveloping group $S^{-1}S$ is amenable because $G$ and $P^{-1}P$ are amenable. Using Proposition~\ref{prop:left-inverse} we conclude that $\Phi_\DD$ is faithful. Finally, $[G:\theta_p(G)] <
\infty$ holds for $p \neq 1_P$ only if $G_0$ is finite and $P \cong \N$.
Indeed, if $p \neq 1_P$, then each element of $\bigoplus_{q \in P\setminus pP}
G_0$ yields a distinct left-coset in $G/\theta_p(G)$. Clearly, this group is
finite if and only if $G_0$ is finite and $P \cong \N$. So if $P$ is not
singly generated or $G_0$ is a countably infinite group, $C^*(S)$ is purely infinite and simple by Theorem~\ref{thm:gxp p.i. simple}.
\end{example}

\noindent A variant of the next example with singly generated $P$ and finite field $\KK$
has been considered in \cite[Example 2.1.4]{CV}.

\begin{example}\label{ex:gxp poly ring}
Let $\KK$ be a countably infinite field and let $G=\KK[T]$ denote the
polynomial ring in a single variable $T$ over $\KK$. We choose non-constant
polynomials $p_i \in \KK[T], i \in I$ for some index set $I$. Multiplying by
$p_i$ defines an endomorphism $\theta_{p_i}$ of $G$ with $[G:\theta_{p_i}(G)]
= |\KK|^{\deg(p_i)}$, where $\deg(p_i)$ denotes the degree of $p_i \in
\KK[T]$. Thus, if we let $P$ be the semigroup generated by all the $p_i$'s, in
notation
\[P := |(p_i)_{i \in I}\rangle,\]
 then the index of $\theta_p(G)$ in $G$ is infinite for all $p \neq 1_P$. It
 is not hard to show that $\theta$ respects the order if and only if $(p_i)
 \cap (p_j) = (p_ip_j)$ holds for the principal ideals whenever $i \neq j$.
 Since every element in $G$ has finite degree, $\bigcap_{p \in P} \theta_p(G)
 = \{1_G\}$ is automatic because the $p_i$ are non-constant. The expectation $\Phi_\DD$ is faithful for the same reason as in Example~\ref{ex:gxp shift}. Thus, provided
 the family $(p_i)_{i \in I}$ has been chosen accordingly, $C^*(S)$ is purely infinite
 and simple.
\end{example}

\noindent We next discuss a class of semigroups $S$ based on a non-commutative group $G$.

\begin{example}\label{ex:gxp F2}
Let $G = \mathbb{F}_2$ be the free group in $a$ and $b$. We define injective
group endomorphisms $\theta_1,\theta_2$ of $G$ by $\theta_1(a) = a^2,
\theta_1(b) = b$, $\theta_2(a) = a, \theta_2(b) = b^2$ and set $P =
|\theta_1,\theta_2\rangle \cong \N^2$. It is clear that the induced action
$\theta$ of $P$ on $G$ respects the order. Additionally, $\bigcap_{p \in P}
\theta_p(G) = \{1_G\}$ is easily checked using the word length coming from
$\{a,a^{-1},b,b^{-1}\}$. To see that $[G:\theta_1(G)] = \infty$ holds, note
that the family $\left((ab)^j\right)_{j\geq 1}$ yields mutually distinct
left-cosets in $G/\theta_1(G)$. The same argument with $ba$ instead of $ab$
shows that $\theta_2(G)$ has infinite index in $G$. Thus, $C^*(S)$ is purely
infinite and simple provided that $\Phi_\DD$ is faithful. One can show that this amounts to amenability of the action $\mathbb{F}_2
\stackrel{\tau}{\curvearrowright} \DD$.
\end{example}

\noindent This example can be viewed as belonging to a larger class, as described next.
The inspiration for these examples was \cite[Example 2.3.9]{Vie}, where
the single endomorphism $\theta_2$ from Example~\ref{ex:gxp F2} on
$\mathbb{F}_2$ is considered.%\vspace*{3mm}

\begin{example}\label{ex:gxp Fn}
For $2 \leq n \leq \infty$, let $\mathbb{F}_n$ be the free group in $n$
generators $(a_k)_{1 \leq k \leq n}$. Fix $1 \leq d \leq n$ and choose for
each $1 \leq i \leq d$ an $n$-tuple $(m_{i,k})_{1 \leq k \leq n} \subset
\N^\times$ such that
\begin{enumerate}[1)]
\item for each $1 \leq i \leq d$, there exists $k$ such that $m_{i,k} > 1$, and
\item for all $1 \leq i,j \leq d, i \neq j$ and $1 \leq k \leq n$, $m_{i,k}$
and $m_{j,k}$ are relatively prime.
\end{enumerate}
Then $\theta_i(a_k) = a_k^{m_{i,k}}$ defines an injective group endomorphism
of $\mathbb{F}_n$ for each $1 \leq i \leq d$. We set $P = |(\theta_i)_{1\leq
i\leq d}\rangle$. Using 2), one can show that the induced action $\theta$ of
$P$ on $G$ respects the order. As in Example~\ref{ex:gxp F2},
$[G:\theta_p(G)]$ is infinite for every $p \neq 1_P$. The requirement
$\bigcap_{p \in P} \theta_p(G) = \{1_G\}$ reduces to
\begin{enumerate}
\item[3)] For each $1 \leq k \leq n$, there exists $1 \leq i \leq d$
satisfying $m_{i,k} > 1$.
\end{enumerate}
So $C^*(S)$ is purely infinite and simple if conditions 1)$-$3) above are satisfied and
$\Phi_\DD$ is faithful. As in Example~\ref{ex:gxp F2}, the latter corresponds to amenability of the action $\mathbb{F}_n \stackrel{\tau}{\curvearrowright} \DD$.\\
\end{example}

\subsection{Semigroups from self-similar actions}~\\
Another large class of right LCM semigroups arises from self-similar actions,
cf. \cite{Law1, LW, BRRW}. We won't be able to say much here, since conditions
(D2) and (D3) are not likely to hold. However, these semigroups
will satisfy condition (D1), and they will satisfy strong effectiveness in the presence of right cancellation. We include these observations here, as well as a description of those semigroups that satisfy (C1).

Let $X$ be a finite alphabet. We write $X^n$ for the set of all words of
length $n$, and $X^*$ for the set of all finite words. We let $\varnothing$ denote the empty word. Under concatenation of
words, $X^*$ is a semigroup (and is nothing more
than $\F_{|X|}^+$). A {\em self-similar action} is a pair $(G,X)$, where $G$
is
a group acting faithfully on $X^*$ and such that for every
$g \in G$ and $x\in X$, there exists a unique $g|_x \in G$ such that
\begin{equation}\label{eq: ss condition}
 g\cdot (xw) = (g\cdot x) (g|_x \cdot w).
\end{equation}
The group element $g|_x$ is called the \emph{restriction} of $g$ to $x$. The
restriction map can be extended iteratively to all finite words, and satisfies
\[
g|_{vw}=(g|_{v})|_{w}, \quad
gh|_{v} = g|_{h\cdot v} h|_{v}, \quad\text{and}\quad
g|_{v}^{-1} = g^{-1}|_{g\cdot v},
\]
for all $g, h\in G$ and $v,w\in X^*$. Moreover, the map $g:X^n \to X^n$ for $n\geq 1$ given
by $w \mapsto g \cdot w$ is bijective. The proof of these properties and much
more can be found in \cite{NekBook}. The Cuntz-Pimsner algebra $\OO(G,X)$ of a
self-similar group has been studied in \cite{Nek1,Nek2}, and the Toeplitz
algebra $\TT(G,X)$ has been studied in \cite{LRRW}.

To each self-similar action $(G,X)$ there exists a semigroup $X^*\bowtie G$,
which
is the set $X^*\times G$ with multiplication given by
\[(x,g)(y,h) = (x(g\cdot y),g|_y h).\]
The semigroup $X^*\bowtie G$ was introduced in \cite{Law1}, and is an example
of a Zappa-Sz\'{e}p product. The $C^*$-algebra $C^*(X^*\bowtie G)$ was studied
in \cite{BRRW}, and was shown to be isomorphic to $\TT(G,X)$.

 Denote $S=X^*\bowtie G$. Then $S$ is right LCM and the principal right ideals are
determined by the element of $X^*$, in the sense that $(w,g)X^*\bowtie
G=(z,h)X^*\bowtie G$ if and only if $w=z$. The identity in $X^*\bowtie G$ is
$(\varnothing,1_G)$, and we have $(X^*\bowtie G)^*=\{\varnothing\}\times G$. For $x\in X$ let
$G_x$ denote the stabiliser subgroup of $x$ in $G$. The map $\phi_x:G_x\to G$ given by $\phi_x(g)=g\vert_x$
is a homomorphism, see for example \cite[Lemma 3.1]{LW}.

We now show that $S$ satisfies (D1), and we determine the
precise conditions under which  the
action $S^* \curvearrowright\JJ(S)$ given by left multiplication on constructible ideals is strongly
effective.

\begin{lemma}\label{lem: self-sim Q2}
Let $(G,X)$ be a self-similar action. Then $X^*\bowtie G$ satisfies (D1) from
Definition~\ref{def: conditions on S for
using D}.
\end{lemma}
\begin{proof}
Let $(\varnothing,h)\in (X^*\bowtie G)^*$ and $(w,g)X^*\bowtie
G\in\JJ(X^*\bowtie G)$ with
\[
(\varnothing,h)(w,g)X^*\bowtie G\cap (w,g)X^*\bowtie G\not=\emptyset.
\]
Then there are $(w',g'),(w'',g'')$ such that
$(\varnothing,h)(w,g)(w',g')=(w,g)(w'',g'')$, and since $w$ and $h\cdot w$
have the same length, this means $w=h\cdot w$. Then
\[
(\varnothing,h)(w,g)X^*\bowtie G=(h\cdot w,h|_wg)X^*\bowtie
G=(w,h|_wg)X^*\bowtie G=(w,g)X^*\bowtie
G.\qedhere
\]
\end{proof}

\noindent We know from  \cite[Proposition 3.11]{LW} that $X^*\bowtie G$ is right
cancellative if and only if $\{w\in X^*:\exists g\in G\setminus\{1_G\}, g\cdot
w=w\text{ and
}g|_w=1_G\}=\emptyset$. This condition also appears in the following result:

\begin{lemma}\label{lem: se action}
Let $(G,X)$ be a self-similar action. Then the action $S^* \curvearrowright\JJ(S)$
given by left multiplication is strongly
effective in the sense of Definition~\ref{def: conditions on S for using D} if
and only if
\[
\{w\in X^*:\exists g\in G\setminus\{1_G\}, g\cdot w=w\text{ and
}g|_w=1_G\}=\emptyset.
\]
\end{lemma}
\begin{proof}
We prove the contrapositive of the forward implication. Suppose $w\in X^*$ and
$g\in G$ with $g\cdot w=w$ and $g|_w=1_G$. Then $(\varnothing,g)\in
(X^*\bowtie G)^*$ and $(w,h)\in S$ satisfy
\[
(\varnothing,g)(w,h)(z,k)X^*\bowtie G=(g\cdot w,g|_wh)(z,k)X^*\bowtie
G=(w,h)(z,k)X^*\bowtie G,
\]
for all $(z,k)\in X ^*\bowtie G$. So the action is not strongly effective.

For the reverse implication, suppose $(\varnothing,g)\in (X^*\bowtie G)^*$ and
$(w,g)\in X^*\bowtie G$. If $g\cdot w\not=w$, then
$(\varnothing,g)(w,h)X^*\bowtie G\not=(w,h)X^*\bowtie G$. If $g\cdot w=w$,
then $g|_w=1_G$ by assumption. Choose $z\in X^*$ such that $g|_w\cdot(h\cdot
z)\not= h\cdot z$. Then
\begin{align*}
(\varnothing,g)(w,h)(z,1_G)X^*\bowtie G &= \big(w(g|_w\cdot (h\cdot
z)),(g|_wh)|_z\big)X^*\bowtie G\\
& \not= (w(h\cdot z),(g|_wh)|_z)X^*\bowtie
G\\
&=(w(h\cdot z),h|_z)X^*\\
&=(w,g)(z,1_G)X^*\bowtie G.
\end{align*}
So the action is strongly effective.
\end{proof}

\noindent We can describe those semigroups $X^*\bowtie G$ that satisfy (C1). Recall  from \cite[page 22]{LW}
 (or \cite{Nek1}) that a self-similar action of $G$ on $X$ is \emph{recurrent}  if the action of $G$ on $X$ is transitive
 and the homomorphism $\phi_x$ is surjective for any $x\in X$. By \cite[Lemma 1.3(8)]{LW}, the last condition is equivalent to
  $\phi_w$ being surjective for all $w\in X^*.$

 \begin{lemma}\label{lem:recurrent}
Let $G$ be a recurrent self-similar action on $X$. Then $S:=X^*\bowtie G$  satisfies (C1). In the converse direction, if $S$ satisfies (C1), then all maps $\phi_x$ for $x\in X$ are surjective.
 \end{lemma}
 \begin{proof} Let $(w,g)\in S$ and $(\varnothing, h)\in S^*$. We will show that there is $(\varnothing, k)\in S^*$
 such that $(w,g)(\varnothing, h)=(\varnothing, k)(w,g)$.  Since $\phi_w$ is surjective, there is $k\in G_w$
 such that $\phi(k)=ghg^{-1}$. In other words, there is $k\in G$ with $k\cdot w=w$ and $k\vert_w=ghg^{-1}$.
 Then
 \[
 (\varnothing, k)(w,g)=(k\cdot w, k\vert_wg)=(w, ghg^{-1}g)=(w,g)(\varnothing, h),
 \]
 showing (C1).

 In the other direction, let $x\in X$ and $g\in G$. From (C1) applied to $(x, g^{-1})\in S$ and $(\varnothing, g)
 \in S^*$, there is $(\varnothing, k)\in S^*$ such that $(x, g^{-1})(\varnothing, g)=(x,e_G)=(\varnothing, k)(x, g^{-1})$, which is
 $(k\cdot x, k\vert_{x}g^{-1})$. This says that $k\in G_x$ and $\phi_x(k)=g$, showing surjectivity for all $x\in X$ (hence for all $x\in X^*$ by \cite[Lemma 1.3(8)]{LW}.)
 \end{proof}

\noindent It would be interesting to know if for the latter class one can prove a uniqueness result using the expectation onto the $C^*$-subalgebra $\CC_I$.

\end{document}